\documentclass[pdflatex,sn-mathphys-num]{sn-jnl}

\usepackage{graphicx}%
\usepackage{multirow}%
\usepackage{amsmath,amssymb,amsfonts}%
\usepackage{amsthm}%
\usepackage{mathrsfs}%
\usepackage[title]{appendix}%
\usepackage{xcolor}%
\usepackage{textcomp}%
\usepackage{manyfoot}%
\usepackage{booktabs}%
\usepackage{algorithm}%
\usepackage{algorithmicx}%
\usepackage{algpseudocode}%
\usepackage{listings}%

\usepackage{mysty_mp}

\begin{document}

\title[Frank-Wolfe with Moreau Envelope Smoothing for Nonsmooth Nonconvex Problems]{Frank-Wolfe with Moreau Envelope Smoothing for Nonsmooth Nonconvex Problems}

\author*[1]{\fnm{Antonio} \sur{Silveti-Falls}}\email{tonys.falls@gmail.com}\equalcont{ALl authors contributed equally.}

\author[2]{\fnm{Cesare} \sur{Molinari}}\email{cecio.molinari@gmail.com}\equalcont{All authors contributed equally.}

\author[3]{\fnm{Zev} \sur{Woodstock}}\email{woodstzc@jmu.edu}\equalcont{All authors contributed equally.}

\affil*[1]{\orgdiv{CVN}, \orgname{CentraleSup\'{e}lec}, \orgaddress{\street{3 Rue Joliot Curie}, \city{Gif-sur-Yvette}, \postcode{91190}, \state{Essonne}, \country{France}}}

\affil[2]{\orgdiv{Department of Mathematics}, \orgname{Università di Genova}, \orgaddress{\street{Via Dodecaneso 35}, \city{Genova}, \postcode{16146}, \state{Liguria}, \country{Italy}}}

\affil[3]{\orgdiv{Department of Mathematics \& Statistics}, \orgname{James Madison University}, \orgaddress{\street{60 Bluestone Drive}, \city{Harrisonburg}, \postcode{22807}, \state{VA}, \country{USA}}}

\abstract{We present and analyze Frank-Wolfe with Moreau Envelope Smoothing (FRAMES) for solving nonsmooth nonconvex constrained optimization problems, taking advantage of iterative smoothing via the Moreau envelope followed by one Frank-Wolfe step per iteration. The problem template we consider encompasses splitting problems with multiple convex constraint sets as well as problems with nonsmooth weakly convex regularizers like the MCP or SCAD penalties. We prove convergence, with rates, for both of these cases under a variety of mild assumptions, including inconsistent problems. Additionally, we highlight a new relationship between the Frank-Wolfe gap for a problem with nonsmooth objective and the Frank-Wolfe gap for a smoothed surrogate problem, demonstrating suboptimality of prior analyses. Numerical experiments are performed for matrix factorization problems and nonconvex quadratic splitting over multiple convex constraint sets, where the improvements in analysis are empirically observed.}


\maketitle

\section{Introduction}\label{sec:intro}
The Frank-Wolfe algorithm \cite{frankwolfe} and its related extensions \cite{guelat1986some,argyriou2014hybrid,lacoste2015,combettes2020boosting} provide an attractive first-order scheme for solving minimization problems posed over a constraint set $\C$ without requiring projections onto it.
Instead of projecting onto $\C$ at each iteration, the Frank-Wolfe algorithm makes use of a \emph{linear minimization oracle} (LMO) to construct its update in a way that guarantees the next iterate remains feasible in $\C$.
This is useful for problems where projections (i.e., quadratic minimization oracles) are computationally expensive but linear minimization oracles are tractable, with recent work precisely detailing the differences in complexity between projection and LMO \cite{combettes2021complexity,woodstock2026high}.

Typically, Frank-Wolfe is applied to problems where the objective function is continuously differentiable, since at each iteration it requires computing a gradient to pass into the linear minimization oracle.
However, many modern problems fit the following, more general constrained composite format
\begin{equation}\tag{P}\label{P}
    \min\limits_{x\in\C\subset \R^n}f(x) + g(Tx),
\end{equation}
where $\C\neq\varnothing$ is the set over which the linear minimization oracle is accessible, $f$ is smooth but not assumed to be convex, $T$ is a linear map, and $g$ is some nonsmooth function whose proximal operator is tractable, typically representing a regularizer or an additional constraint set $\D$ whose projection is cheap to compute.
Examples fitting into \eqref{P} include inverse problems when $f$ models the data-fidelity of the recovered solution and $g$ is some regularizer, e.g., the MCP or SCAD penalties, or the indicator of an additional convex constraint set, like the nonnegative orthant.
Directly applying Frank-Wolfe to such problems is not simple; for instance, replacing the gradient with a subgradient and applying Frank-Wolfe does not necessarily lead to convergence, even on simple functions like $g(x)=\max(x_1,x_2)$ \cite{yurtsever2018conditional}.

There exist nonsmooth extensions of the Frank-Wolfe algorithm \cite{argyriou2014hybrid,yurtsever2018conditional,yurtsever2019conditional,silveti2020generalized,woodstock2025splitting} that use the Moreau envelope to smooth the function $g$.
The existing analyses in these cited works all assume that $g$ is convex. Furthermore, all of them except \cite{woodstock2025splitting} assume that $f$ is convex, thus excluding nonconvex instances of \eqref{P} like matrix factorization problems.
For these smooth nonconvex problems, first-order stationarity is typically measured using the Frank-Wolfe gap (an analog of the norm of the gradient for unconstrained problems), detailed in Section~\ref{sec:gaps}.
However, the compatibility of the Frank-Wolfe gap with the Moreau envelope is much less studied.
While \cite{woodstock2025splitting} established asymptotic subsequential convergence to a stationary point of \eqref{P} is possible when $g$ is the indicator function of a specific linear subspace, in the general setting of \eqref{P}, it is not obvious how the Frank-Wolfe gap associated to a surrogate smoothed problem will certify stationarity of the original problem. 

\paragraph{Approach} 
As in \cite{argyriou2014hybrid,yurtsever2018conditional,silveti2020generalized,woodstock2025splitting}, we replace the nonsmooth term $g$ by its Moreau envelope with a vanishing smoothing parameter. More precisely, at iteration $k$ we consider the smooth surrogate
\begin{equation*}
    \Phi_k(x) := f(x)+g^{\beta_k}(Tx),
\end{equation*}
where $\beta_k>0$ decreases to zero, and we apply one Frank-Wolfe step to $\Phi_k$ over the set $\C$. This construction preserves the projection-free nature of the method with respect to $\C$, since each update only requires a linear minimization oracle over $\C$, while the nonsmooth term is handled through the proximity operator of $g$. The use of a decreasing smoothing parameter is essential: for fixed $\beta$, the algorithm would only solve a smoothed approximation of \eqref{P}, whereas letting $\beta_k\to0$ forces the surrogate objectives to return to the original problem. At the same time, this creates a tradeoff, since smaller values of $\beta_k$ typically lead to worse smoothness constants for $\Phi_k$. The main goal of this paper is therefore to analyze this tradeoff and to determine when convergence of the Frank-Wolfe gaps associated with the smoothed objectives yields meaningful stationarity guarantees for the original nonsmooth problem \eqref{P} without assuming convexity.

\paragraph{Contributions} We present a Moreau-envelope Frank-Wolfe framework for the nonsmooth composite problem \eqref{P}, together with convergence guarantees that distinguish between progress on the smoothed surrogate problems and stationarity of the original nonsmooth problem. Our main contributions are as follows.

\begin{itemize}
    \item \textbf{Smoothed Frank-Wolfe gap convergence with vanishing smoothing.}
    We prove that, despite the objective $\Phi_k$ changing at every iteration, the average and best Frank-Wolfe gaps associated with the smoothed objectives admit explicit convergence rates. In particular,  for $p,q\in\left]0,1\right[$, step sizes $\gamma_k=(k+1)^{-p}$ and smoothing schedules $\beta_k=\beta_0(k+1)^{-q}$,
    \thmref{thm:rates} establishes an $\mathcal{O}(N^{-\min\{p-q,\ 1-p-q\}})$ bound on the average smoothed Frank-Wolfe gap, and hence on the best smoothed gap among the first $N$ iterates. This result extends the usual smooth nonconvex Frank-Wolfe gap analysis to a setting with a nonsmooth composite term, including both indicator and Lipschitz weakly convex cases, a linear composition, and a vanishing smoothing parameter. Remark~\ref{rmk:optimal_power_balance_nonsmooth} later explains that the choice $p=1/2$ and $q=1/4$ optimizes the final nonsmooth stationarity rate to $\mathcal{O}(N^{-1/4})$.
    \item {\bf Nonsmooth stationarity guarantee.} We also show in \thmref{thm:indicator_subsequential_stationarity} and \thmref{thm:lipschitz_subsequential_stationarity} that convergence of the smoothed Frank-Wolfe gaps along a subsequence guarantees cluster points of that subsequence are stationary points, which provides a generalization of \cite{woodstock2025splitting} beyond the setting when $g$ is the indicator function of a specific linear subspace. In general, one cannot expect convergence of the full sequence of iterates, since even in the smooth setting Frank-Wolfe iterates need not converge \cite{bolte2024iterates}. We also show in \thmref{thm:inconsistent_indicator_subsequential} that, if the underlying problem is inconsistent, the algorithm still produces a stationary point of a relaxed problem that minimizes the distance between the constraint sets.
    \item \textbf{Transfer from smoothed gaps to nonsmooth stationarity certificates.}
    We then analyze when small smoothed Frank-Wolfe gaps imply meaningful progress for the original nonsmooth problem \eqref{P}. For the indicator case $g=\iota_{\D}$, Lemma~\ref{lem:indicator_finite_time_transfer} characterizes relationships between the smoothed gap, the feasibility violation $\operatorname{dist}_{\D}(Tx_k)$, the smoothing parameter $\beta_k$, and the Frank-Wolfe gap over the true feasible set $\widetilde{\C}:=\C\cap T^{-1}(\D)$. For the Lipschitz weakly convex case, our gap-transfer result, \lemref{lem:lipschitz_gap_transfer}, bounds an appropriate nonsmooth Frank-Wolfe gap in terms of the smoothed gap and the smoothing parameter. These estimates lead to explicit rates on the nonsmooth Frank-Wolfe gap in Theorems~\ref{thm:indicator_nonsmooth_gap_rate} and \ref{thm:nonsmooth_gap_rate}, clarifying how the decay of $\beta_k$ affects stationarity guarantees for the original problem \eqref{P}.
    
    These bounds are also independent of our algorithm, and can be applied to any problem of the form \eqref{P} that uses Moreau smoothing. For instance, Remark~\ref{rmk:log_is_slower} demonstrates how the smoothed convergence rates of \cite{woodstock2025splitting,halbey2025efficient} translate to nonsmooth convergence rates, revealing their step size and smoothing schedules to be suboptimal. 
    
    \item \textbf{Numerical evidence and smoothing-parameter behavior.}
    We illustrate the theory on examples involving nonnegative matrix factorization, trend-filtered matrix factorization with nonconvex nonsmooth penalties, and a nonconvex splitting problem over intersecting constraint sets in \secref{sec:experiments}. These experiments demonstrate the practical effect of the smoothing schedule and the sensitivity to the initial smoothing parameter. The splitting experiment is designed in such a way that the nonsmooth Frank-Wolfe gap is tractable, and the difference between smoothed gap convergence and progress on the original nonsmooth stationarity measures is shown.
\end{itemize}

\paragraph{Outline} We finish this section with a discussion of related work before moving on to \secref{sec:setup_algorithm}, where we formally state the standing assumptions on \eqref{P} and present the main algorithm, \FRAMES. In \secref{sec:preliminary}, we collect the elementary estimates on Moreau envelopes and the smoothed objectives that are used throughout the convergence analysis. In \secref{sec:gaps}, we introduce the different Frank-Wolfe gaps and stationarity certificates that arise for the smoothed problem, the indicator case, and the Lipschitz weakly convex case. In \secref{sec:smoothed_convergence}, we prove convergence rates for the smoothed Frank-Wolfe gaps generated by the algorithm and provide guarantees for subsequential stationary points of \eqref{P}. In \secref{sec:gap_transfer}, we relate the convergence of these smoothed gaps to stationarity certificates for the original, nonsmooth problem \eqref{P}, including both finite-time gap-transfer results (yielding explicit convergence rates) and subsequential stationarity guarantees for inconsistent problems. Finally, in \secref{sec:experiments}, we illustrate the behavior of the method on several numerical examples and highlight the effect of the smoothing schedule on the certificate of stationarity for \eqref{P}.

\subsection{Related Work}
We separate the discussion of related work into categories based on three features: the use of a linear minimization oracle over $\C$, the presence of the nonsmooth term $g\circ T$, and the notion of stationarity used in the results.

\paragraph{Smooth Nonconvex Frank-Wolfe}
For smooth constrained problems, i.e., \eqref{P} with $g=0$, Frank-Wolfe has been studied in the nonconvex setting already, with the notable work of \cite{lacoste2015} establishing an $\mathcal{O}(N^{-1/2})$ convergence rate for the Frank-Wolfe gap using either a short step or a line search step size.

\paragraph{Moreau Envelope Smoothing and Convex Composite Frank-Wolfe}
Several works considering convex composite problems used the Moreau envelope (sometimes under the name homotopy-based smoothing) in conjunction with the linear minimization oracle over $\C$ as in Frank-Wolfe. In \cite{argyriou2014hybrid,yurtsever2018conditional}, the nonsmooth term in a composite objective is replaced by a Moreau envelope with a decreasing smoothing schedule, with one step of Frank-Wolfe applied to the smoothed surrogate problem at each iteration; \cite{argyriou2014hybrid} considered only Lipschitz nonsmooth terms while \cite{yurtsever2018conditional} considered both Lipschitz nonsmooth terms and indicator functions. Augmented-Lagrangian variants for handling affine constraints like $Ax=b$ separately from $\C$ were developed in \cite{yurtsever2019conditional,silveti2020generalized}, with inexact and stochastic versions appearing in \cite{silveti2021inexact}. Stochastic composite Frank-Wolfe methods were also studied in \cite{loca}.

Other approaches based on the Moreau envelope with linear minimization oracle over $\C$ for nonsmooth convex optimization include MOPES and MOLES \cite{thekumparampil2020projection,thekumparampil2020optimal}. These methods also use Moreau smoothing, but their goal is convex suboptimality or regret, and the LMO-based variant approximates projection-type operations through an inner Frank-Wolfe routine. They do not address nonconvex stationarity for the composite problem \eqref{P}.

All the algorithms studied in these papers are similar to \FRAMES\ in the sense that they combine a linear minimization oracle over $\C$ with the proximal operator of $g$. However, these analyses all require convexity and provide guarantees on functional gaps, primal-dual gaps, or regret. In contrast, we consider \eqref{P} with nonconvex $f$ and nonsmooth weakly convex $g$, for which the appropriate notion of first-order stationary point is better captured by the Frank-Wolfe gap.

\paragraph{Frank-Wolfe Splitting}
The closest nonconvex predecessor to our work is \cite{woodstock2025splitting}, which studies a Frank-Wolfe method for minimizing a smooth nonconvex function over a nonempty intersection of compact convex sets,
\begin{equation*}
    \min_{x\in\bigcap_{i=1}^m\C_i} f(x).
\end{equation*}
Their approach lifts the problem to the product space $\C_1\times\cdots\times\C_m$ and enforces consensus by taking $g$ to be the indicator of the diagonal subspace (\`a la Pierra \cite{Pier84}), yielding a special case of \eqref{P}. By smoothing this indicator and applying Frank-Wolfe to the resulting surrogate problems, \cite{woodstock2025splitting} proves a convergence rate for averaged smoothed Frank-Wolfe gaps in the nonconvex setting. A follow-up work \cite{halbey2025efficient} improves the corresponding smoothed-gap rate through a modified step-size choice. These splitting methods are highly relevant to our work, but they treat a specific consensus-constraint structure. Our analysis keeps the same Moreau-smoothing philosophy while allowing more general nonlinear constraints and Lipschitz weakly convex nonsmooth terms, and we explicitly relate the smoothed Frank-Wolfe gaps to stationarity certificates for the original nonsmooth problem.

\paragraph{Variable Smoothing and Weak Convexity}
Variable smoothing methods based on the Moreau envelope have also been studied outside the Frank-Wolfe setting. Most relevant for us is \cite{bohm2021variable}, which considers unconstrained composite problems involving a weakly convex nonsmooth term and a decreasing Moreau smoothing parameter. The estimates used there are similar in spirit to several of the Moreau-envelope estimates used in our analysis. The main difference is algorithmic and geometric: \cite{bohm2021variable} does not work in a projection-free Frank-Wolfe setting, nor does it analyze stationarity through Frank-Wolfe gaps. In contrast, our setting \eqref{P} allows for constrained problems and our method preserves the linear-minimization-oracle geometry over $\C$, uses only the proximal operator of $g$, and studies how stationarity of the smoothed surrogate transfers back to stationarity certificates for the original nonsmooth composite problem \eqref{P}.

There is also a broader literature on projection-free methods for convex or weakly smooth composite problems, including conditional-gradient methods under H\"older or weak smoothness assumptions \cite{lan2016conditional,ouyang2023universal,ito2023parameter}. The convergence guarantees in these works require convexity and are not designed to address nonsmooth nonconvex stationarity for the composite problem \eqref{P}.

\paragraph{Nonsmooth Frank-Wolfe Methods Without Smoothing}
A complementary line of work develops Frank-Wolfe-type methods for nonsmooth problems without using Moreau smoothing. \cite{de2023short} studies a Frank-Wolfe method for upper-$C^{1,\alpha}$ functions over compact convex sets and proves rates toward Clarke stationarity. Although this setting is close to that of \eqref{P}, the distinction is that \cite{de2023short} relies on an upper-$C^{1,\alpha}$ structure and a suitable generalized gradient selection, whereas our approach exploits the composite form $f+g\circ T$, prox access to $g$, and Moreau smoothing of the nonsmooth outer term.

Another recent direction is Frank-Wolfe for abs-smooth functions \cite{kreimeier2024frank}, where the algorithm uses an abs-linearization of the objective and a generalized Frank-Wolfe gap. This provides an alternative to smoothing and can recover smooth-like rates for the corresponding generalized gap. However, the method relies on the specialized abs-smooth structure and requires solving a more involved subproblem rather than a standard linear minimization oracle over $\C$. Thus, it is complementary to the present approach, which keeps the usual Frank-Wolfe linear oracle and instead handles nonsmoothness through the Moreau envelope of $g$.

In the convex nonsmooth setting, several projection-free approaches also avoid Moreau smoothing. Deterministic nonsmooth Frank-Wolfe methods with coreset guarantees were studied in \cite{ravi2019deterministic}, while methods based on uniform affine approximations for separable nonsmooth convex functions were developed in \cite{cheung2018solving}. More recently, \cite{asgari2024nonsmooth} studied nonsmooth projection-free convex optimization with functional constraints using subgradients and a separation-type scheme. These works show that nonsmooth projection-free optimization can be approached without smoothing, but their guarantees are for convex problems or rely on specialized structure that is different from the general nonconvex composite model \eqref{P}.

Finally, \cite{mazanti2025nonsmooth} proposes a nonsmooth Frank-Wolfe algorithm through a dual cutting-plane perspective. Their approach is quite different from the Moreau-envelope methods discussed above: it reinterprets the fully corrective Frank-Wolfe algorithm as dual to a cutting-plane method with partial linearization, and then uses this dual viewpoint to develop a nonsmooth Frank-Wolfe-type scheme. This provides an important recent no-smoothing alternative, although it requires convexity and is algorithmically distinct from the single-loop Moreau-smoothed Frank-Wolfe method studied here.

\section{Problem Setup and Algorithm}\label{sec:setup_algorithm}

\subsection{Standing Assumptions}

We work under the following assumptions on \eqref{P} throughout the paper.

\begin{assumption}[Smooth term]\label{ass:f}
The function $f:\R^n\to\R$ is continuously differentiable on an open set containing $\C$, and its gradient is Lipschitz-continuous on $\C$ with constant $\gradfconstant>0$.
\end{assumption}

\begin{assumption}[Linear map]\label{ass:T}
The operator $T:\R^n\to\R^m$ is linear.
\end{assumption}

\begin{assumption}[Frank-Wolfe feasible set]\label{ass:C}
The nonempty set $\C\subset\R^n$ is compact and convex, with diameter $\dc:=\sup_{x,y\in\C}\|x-y\|_2<+\infty$.
\end{assumption}

\begin{assumption}[Nonsmooth term]\label{ass:g}
The function $g:\R^m\to\left]-\infty,+\infty\right]$ satisfies one of the following:
\begin{enumerate}[label=(\Roman*)]
    \item \label{ass:g=indicator}
    $g=\iota_{\D}$ for some nonempty closed convex set $\D\subset\R^m$.
    \begin{enumerate}[label=(\alph*)]
        \item \label{ass:g=indicator_nonempty}(optional) Additionally, it holds $\D\cap T(\C)\neq \varnothing$.
        \item \label{ass:g=indicator_nonempty_ri}(optional) Additionally, it holds $\operatorname{ri}(\D)\cap \operatorname{ri}(T(\C))\neq\varnothing$.
    \end{enumerate}
    \item \label{ass:g=lipschitz}
    $g\colon\R^m\to\R$ is $L_g$-Lipschitz-continuous on $\R^m$ and $\rho$-weakly convex for some $\rho\geq0$, i.e., $g+\frac{\rho}{2}\|\cdot\|_2^2$ is convex.
\end{enumerate}
\end{assumption}
 Without loss of generality, Assumption~\ref{ass:g} provides that $g$ is $\rho$-weakly convex; under \assref{ass:g}\ref{ass:g=indicator} we use the convention $\rho=0$ with $\rho^{-1}=+\infty$. We assume access to three computational primitives: the gradient $\nabla f$, the linear minimization oracle (LMO), given some vector $v$, returns a point in $\arg\min_{s\in\C}\langle v,s\rangle$, and the proximal operator of $g$, $\prox_{\beta g}(y)
    :=
    \argmin_{u\in\R^m}
    \left\{
        g(u)+\frac{1}{2\beta}\norm{u-y}{2}^2
    \right\}.$
Under Assumption~\ref{ass:g}\ref{ass:g=indicator}, for every $\beta>0$, $\prox_{\beta g}=\proj_{\D}$. Open-source repositories for proximal operators and LMOs can be found, e.g., at \cite{FW.jl,prox-op.net}.

\subsection{Notation}
\label{sec:notation}

For general background see, e.g., \cite{rockafellar1998variational}.
We denote $\N:=\{0,1,2,\ldots\}$ and $\N^*:=\N\setminus\{0\}$. For a proper function $g:\R^m\to\left]-\infty,+\infty\right]$, its \emph{domain} is $\dom(g):=\{x\in\R^m:g(x)<+\infty\}$.
The \emph{indicator function} of a set $S$ is $\iota_S(x):=
        0 \;\;\text{if }x\in S; \text{and}
        +\infty \;\;\text{if }x\notin S.$
For a closed set $S$, $\dist_{S}(y):=\inf_{u\in S}\norm{y-u}{2}$ and $\dist_{S}^2(y):=\para{\dist_{S}(y)}^2$ for the \emph{distance} (or its square) to $S$. The \emph{projection} onto a nonempty closed convex set $\D$ is denoted by $\proj_{\D}$. Since $\C$ is compact and $T$ is linear, $T(\C)$ is compact, hence the \emph{diameter}  $\dtc:=\sup_{x,y\in\C}\norm{Tx-Ty}{2}$ is finite.
Under Assumption~\ref{ass:g}\ref{ass:g=indicator}, the \emph{one-sided excess} of $T(\C)$ from $\D$,
$
    e_{\D,T(\C)}
    :=
    \max_{z\in T(\C)}\dist_{\D}(z),
$
is finite since $T(\C)$ is compact by \assref{ass:T} and \assref{ass:C}. Moreover, under the additional hypotheses in  either \assref{ass:g}\ref{ass:g=indicator}\ref{ass:g=indicator_nonempty} or \assref{ass:g}\ref{ass:g=indicator}\ref{ass:g=indicator_nonempty_ri}, the intersection $\D\cap T(\C)$ is nonempty and therefore
$
    e_{\D,T(\C)}\leq \dtc.
$
Finally, since $f$ is continuously differentiable on an open set containing the compact set $\C$, the quantity
\begin{equation*}
    \fconstant:=\sup_{x\in\C}\norm{\nabla f(x)}{2}
\end{equation*}
is finite. Hence, for all $x,y\in\C$,
$
    f(x)-f(y)\leq \fconstant\norm{x-y}{2}\leq \fconstant\dc.
$
When $g$ is Lipschitz-continuous, we will denote the \emph{Clarke subdifferential} of $g$ at $x$ by $\partial g(x)$. When $g$ is convex and Lipschitz-continuous, this coincides with the classical convex subdifferential. When $g=\iota_{S}$ is the indicator function for a closed convex set $S$ and $x\in S$, $\partial g$ will denote the \emph{normal cone}, defined as
$
    N_S(x):=
    \left\{v:\ip{v,s-x}{}\leq0\quad\forall s\in S\right\},
$
with $N_S=\varnothing$ if $x\not\in S$. Here and throughout, $\operatorname{ri}$ denotes \emph{relative interior} and $T(\C):=\{Tx\colon x\in\C\}$.
For $\beta>0$, the Moreau envelope of $g$ is 
\begin{equation}
\label{e:moreau}
    g^\beta(y)
    :=
    \min\limits_{u\in\R^m}
        g(u)+\frac{1}{2\beta}\norm{u-y}{2}^2.
\end{equation}
This work only considers the choice $\beta<\rho^{-1}$ which guarantees $\prox_{\beta g}$ is unique. With this choice of $\beta$, under \assref{ass:g}, $g^{\beta}$ is also continuously differentiable \cite[Corollary~3.4]{hoheisel2010proximal}, and $
    \nabla g^\beta(y)
    =
    \frac{1}{\beta}
    \left(
        y-\operatorname{prox}_{\beta g}(y)
    \right)$.

\subsection{Algorithm}

At iteration $k$, we smooth $g$ using the Moreau envelope to get the following smoothed objective and gradient
\begin{equation}
\label{e:phi}
    \Phi_k(x)
    :=
    f(x)+g^{\beta_k}(Tx)\;\text{with}\;\;
    \nabla\Phi_k(x)
    =
    \nabla f(x)
    +
    \frac{1}{\beta_k}
    T^*\left(
        Tx-\operatorname{prox}_{\beta_k g}(Tx)
    \right).
\end{equation}
The \FRAMES\ algorithm applies one Frank-Wolfe step to the smoothed objective $\Phi_k$ (only smoothing $g$ rather than $g\circ T$) over $\C$ at each iteration. Since $\C$ is convex and the step size $\gamma_k\in\left]0,1\right]$, the update in Step~\ref{a:update} guarantees $x_k\in\C$ for all $k$ whenever $x_0\in\C$. However, under Assumption~\ref{ass:g}\ref{ass:g=indicator}, the same is not guaranteed for the constraint $T^{-1}(\D)$; the iterates may not satisfy $Tx_k\in\D$ for any finite $k$.

\begingroup
\renewcommand{\figurename}{Algorithm}
\begin{figure}[ht]
\label{algorithm}
\centering
\caption{\textbf{Fra}nk-Wolfe with \textbf{M}oreau \textbf{E}nvelope
\textbf{S}moothing (\textsc{\textbf{Frames}})}
\fbox{\parbox{0.95\textwidth}{%
\medskip
\textbf{Input:} $x_0\in\C$, step sizes $\seq{\gamma_k}\subset {]0,1]}$,
and smoothing schedule $\seq{\beta_k}\subset {]0,\rho^{-1}[}$
\medskip
\textbf{For} $k=0,1,2,\ldots$ \textbf{do}
\begin{enumerate}[leftmargin=2em]
    \item Compute the gradient of the smoothed objective
    \begin{equation*}
        \nabla \Phi_k(x_k)
        =
        \nabla f(x_k)
        +
        \frac{1}{\beta_k}
        T^*\!\left(
            Tx_k-\operatorname{prox}_{\beta_k g}(Tx_k)
        \right).
    \end{equation*}
    \item 
    \label{a:lmo} 
    Compute the LMO for this direction
    \begin{equation*}
        s_k\in\arg\min_{s\in\C}
        \langle \nabla\Phi_k(x_k),s\rangle.
    \end{equation*}
    \item 
    \label{a:update} 
    Update the iterate
    \begin{equation*}
        x_{k+1}=x_k+\gamma_k(s_k-x_k).
    \end{equation*}
\end{enumerate}
\textbf{End for}
}}
\end{figure}
\endgroup

\begin{remark}
Even if $\prox_{\beta g}$ is available for $\beta>0$, unless $T$ has special structure (like orthogonality), evaluating $\prox_{g\circ T}$ is typically not possible in closed-form. Hence, by smoothing only the outer function $g$ instead of the composition $g\circ T$, \FRAMES\ only needs access to $\prox_{\beta g}$ for $\beta\in\left]0,\rho^{-1}\right[$ to compute $\nabla \Phi_k(\xk)$ at each iteration.
\end{remark}

The main results in this paper will be stated for the open-loop schedules $\gamma_k=(k+1)^{-1/2}$ and $\beta_k=\beta_0(k+1)^{-1/4}$, where $\beta_0\in\left]0,\rho^{-1}\right[$ in the Lipschitz weakly convex case, with the convention $\rho^{-1}=+\infty$ when $\rho=0$. Both algorithms appearing in \cite{woodstock2025splitting, halbey2025efficient} arise as special cases of \FRAMES\ under the special case of \assref{ass:g}\ref{ass:g=indicator} when $\D$ is a specific linear subspace. These suggested schedules of $\gamma_k=\mathcal{O}(k^{-1/2})$ and $\beta_k=\mathcal{O}(1/\log(k))$ will still yield convergence but our analysis shows that this will lead to inferior performance compared to the main schedules proposed in this work. Some intermediate results are stated under the following, more general assumptions.

\begin{assumption}\label{ass:gamma}
    The step size sequence $\seq{\gamma_k} \subset\, ]0,1]$ is nonincreasing and $\underset{k\to\infty}{\lim}\gamma_k=0$.
\end{assumption}

\begin{assumption}\label{ass:beta}
    The smoothing parameter sequence $\seq{\beta_k} \subset ]0,\rho^{-1}[$ is nonincreasing and $\underset{k\to\infty}{\lim}\beta_k=0$.
\end{assumption}


\section{Preliminaries and Basic Estimates}\label{sec:preliminary}

The main purpose of this section is to record how the Moreau envelope behaves as the smoothing parameter varies, and to derive uniform bounds on $\Phi_k(x)=f(x)+g^{\beta_k}(Tx)$.

\subsection{Moreau Envelope Estimates}

\begin{proposition}[Moreau envelope calculus]\label{prop:moreauProperties}
Let $g$ satisfy \assref{ass:g}, and let $\beta\in\left]0,\rho^{-1}\right[$. Then $g^\beta$ is continuously differentiable on $\R^m$ and
$(    \forall y\in\R^m)\qquad
    \nabla g^\beta(y)
    =
    \frac{1}{\beta}\para{y-\prox_{\beta g}(y)}.
$
Moreover, $\nabla g^\beta$ is Lipschitz-continuous with constant
$
    L_{g,\beta}
    :=
    \max\left\{\beta^{-1},\frac{\rho}{1-\rho\beta}\right\}\leq \frac{1}{\beta(1-\rho\beta)}.
$
In particular, for every $y\in\R^m$, under \assref{ass:g}\ref{ass:g=indicator},
$
    g^\beta(y)=\frac{1}{2\beta}\dist_{\D}^2(y)$ and 
  $
    \nabla g^\beta(y)=\frac{1}{\beta}\para{y-\proj_{\D}(y)};$
under \assref{ass:g}\ref{ass:g=lipschitz},
$
    \norm{\nabla g^\beta(y)}{2}\leq L_g.
$
\end{proposition}

\begin{proof}
The differentiability of $g^\beta$, the gradient formula, and the Lipschitz estimate for $\nabla g^\beta$ follow from \cite[Corollary 3.4]{hoheisel2010proximal} and the choice $\beta\in\left]0,\rho^{-1}\right[$. The formula under \ref{ass:g=indicator} follows from the identity $\prox_{\beta\iota_{\D}}=\proj_{\D}$. Finally, under \ref{ass:g=lipschitz}, the optimality condition for $\prox_{\beta g}$ yields
\begin{equation*}
    \nabla g^\beta(y)
    =
    \frac{1}{\beta}\para{y-\prox_{\beta g}(y)}
    \in
    \partial g(\prox_{\beta g}(y)),
\end{equation*}
and the Clarke subgradients of an $L_g$-Lipschitz function have norm at most $L_g$ \cite{clarke1990optimization}.
\end{proof}

\begin{proposition}[Lipschitz displacement bound]\label{prop:moreauDisplacementBound}
Suppose $g$ satisfies \assref{ass:g}\ref{ass:g=lipschitz}. Then, for every $\beta\in\left]0,\rho^{-1}\right[$ and every $y\in\R^m$,
\begin{equation*}
    \norm{y-\prox_{\beta g}(y)}{2}
    \leq
    \beta L_g.
\end{equation*}
Moreover, if $\seq{\beta_k}$ satisfies \assref{ass:beta}, then for every sequence  $\seq{x_k}\subset\R^n$,
$
    \norm{Tx_k-\prox_{\beta_k g}(Tx_k)}{2}
    \leq
    \beta_k L_g
    \longrightarrow 0.
$
\end{proposition}
\begin{proof}
Using the expression and bounds for $\nabla g^{\beta}$ given in \propref{prop:moreauProperties},
\begin{equation}
\label{e:32}
    \norm{y-\prox_{\beta g}(y)}{2}
    =
    \beta\norm{\nabla g^\beta(y)}{2}
    \leq
    \beta L_g.
\end{equation}
The second claim follows by taking $y=Tx_k$ in \eqref{e:32} and using $\beta_k\to0$ from \assref{ass:beta}.
\end{proof}

\begin{proposition}[Bound on $\Id-\prox_{\beta g}$ over $T(\C)$]\label{prop:idProxBound}
Let Assumptions~\ref{ass:T}, \ref{ass:C}, and \ref{ass:g} hold. For $\beta\in\left]0,\rho^{-1}\right[$, 
$
    \max_{z\in T(\C)}\norm{z-\prox_{\beta g}(z)}{2}<+\infty.    
$
Moreover, if $\seq{\beta_k}$ satisfies \assref{ass:beta}, then for all $k\in\N^*$ it holds
\begin{equation}
\label{e:bprox}
    \max_{z\in T(\C)}\norm{z-\prox_{\beta_k g}(z)}{2}
    \leq B_{\prox} :=
    \begin{cases}
        e_{\D,T(\C)}
        & \mbox{if $g$ satisfies \assref{ass:g}\ref{ass:g=indicator},}\\
        \beta_0L_g
        & \mbox{if $g$ satisfies \assref{ass:g}\ref{ass:g=lipschitz}.}
    \end{cases}
\end{equation}
\end{proposition}

\begin{proof}
Since $T(\C)$ is compact and $\prox_{\beta g}$ is continuous for $\beta\in\left]0,\rho^{-1}\right[$ since $\prox_{\beta g}=\Id - \beta\nabla g^\beta$, the map
$    z\mapsto \norm{z-\prox_{\beta g}(z)}{2}$
attains its maximum on $T(\C)$. This proves finiteness.
Under Assumption~\ref{ass:g}\ref{ass:g=indicator}, for every $\beta>0$, $\prox_{\beta g}=\proj_{\D}$ and
$
    \max_{z\in T(\C)}
    \norm{z-\proj_{\D}(z)}{2}
    =
    e_{\D,T(\C)}.
$
Under \assref{ass:g}\ref{ass:g=lipschitz},  \propref{prop:moreauDisplacementBound} and Assumption~\ref{ass:beta} provide
   $ \max_{z\in T(\C)}\norm{z-\prox_{\beta g}(z)}{2}
    \leq
    \beta_k L_g
    \leq
    \beta_0 L_g$
because $\seq{\beta_k}$ is nonincreasing.
\end{proof}

\subsection{Estimates for the Smoothed Objectives}

\begin{proposition}[Descent lemma applied to $\Phi_k$]\label{prop:phiDescent}
Let Assumptions~\ref{ass:f}, \ref{ass:T}, \ref{ass:C}, and \ref{ass:g} hold. For every $\beta\in\left]0,\rho^{-1}\right[$, define
$
    \Phi_\beta(x):=f(x)+g^\beta(Tx)
$ and set 
\begin{equation}\label{eq:Lbeta}
    L_\beta
    :=
    \gradfconstant
    +
    \norm{T}{\op}^2
    \max\left\{\beta^{-1},\frac{\rho}{1-\rho\beta}\right\}.
\end{equation}
Then, for all $x,y\in\C$,
\begin{equation*}
    \Phi_\beta(y)
    \leq
    \Phi_\beta(x)
    +
    \ip{\nabla f(x)+T^*\nabla g^\beta(Tx),y-x}{}
    +
    \frac{L_\beta}{2}\norm{y-x}{2}^2.
\end{equation*}
\end{proposition}

\begin{proof}
By \assref{ass:f}, $\nabla f$ is $\gradfconstant$-Lipschitz on $\C$. By \propref{prop:moreauProperties}, $\nabla g^\beta$ is Lipschitz-continuous with constant
$
    \max\{\beta^{-1},\frac{\rho}{1-\rho\beta}\}.
$
Therefore, $\nabla\Phi_\beta$ is Lipschitz-continuous on $\C$ with constant $L_\beta$. The standard descent lemma for $C^{1,1}$ functions then gives the claim \cite{CGsurvey}.
\end{proof}

Proposition~\ref{prop:phiDescent} provides an estimated smoothness constant of $\Phi_k$ in \eqref{eq:Lbeta}:
\begin{equation}\label{eq:Lk}
    L_k
    :=
    \gradfconstant
    +
    \norm{T}{\op}^2
    \max\left\{\beta_k^{-1},\frac{\rho}{1-\rho\beta_k}\right\}.
\end{equation}

\begin{lemma}[Variation of $\Phi_k$ with respect to $k$]\label{lem:phiEstimate}
Let Assumptions~\ref{ass:f}, \ref{ass:T}, \ref{ass:C}, \ref{ass:g}, and \ref{ass:beta} hold, let $\Phi_k$ be given by \eqref{e:phi}, and let $B_{\prox}$ be given by \eqref{e:bprox}. Then,
\begin{equation}
\label{e:lem-var-phi}
(\forall k\in\N)(\forall x\in \C)\qquad
    \Phi_{k+1}(x)-\Phi_k(x)
    \leq
    \frac{1}{2}\para{\beta_{k+1}^{-1}-\beta_k^{-1}}B_{\prox}^2.
\end{equation}
\end{lemma}
\begin{proof}
For every $x\in\C$, $\Phi_{k+1}(x)-\Phi_k(x) = g^{\beta_{k+1}}(Tx)-g^{\beta_k}(Tx)$.

Suppose that \assref{ass:g}\ref{ass:g=indicator} holds. Then
$
    g^{\beta_k}(Tx)
    =
    \frac{1}{2\beta_k}\dist_{\D}^2(Tx),
$
so since $\dist_{\D}(\cdot)\leq e_{\D,T(\C)} =B_{\prox}$ (see Section~\ref{sec:notation} and \eqref{e:bprox}), Assumption~\ref{ass:beta}
yields
\begin{equation*}
    \Phi_{k+1}(x)-\Phi_k(x)
    =
    \frac{1}{2}\para{\beta_{k+1}^{-1}-\beta_k^{-1}}\dist_{\D}^2(Tx)
     \leq
    \frac{1}{2}\para{\beta_{k+1}^{-1}-\beta_k^{-1}}B_\prox^2.
\end{equation*}

Suppose now that \assref{ass:g}\ref{ass:g=lipschitz} holds.
By \assref{ass:beta}, $\beta_k\geq \beta_{k+1}$, so for every $y\in\R^m$, $g^{\beta_{k+1}}(y)\geq g^{\beta_k}(y)$.
By substituting $\prox_{\beta_k g}(y)$ in the definition of $g^{\beta_{k+1}}(y)$ in \eqref{e:moreau}, we find
\begin{equation*}
    \begin{aligned}
    g^{\beta_{k+1}}(y)-g^{\beta_k}(y)
    &\leq
    \left(
        g(\prox_{\beta_k g}(y))+\frac{1}{2\beta_{k+1}}\norm{y-\prox_{\beta_k g}(y)}{2}^2
    \right)\\
    &\quad\quad-
    \left(
        g(\prox_{\beta_k g}(y))+\frac{1}{2\beta_k}\norm{y-\prox_{\beta_k g}(y)}{2}^2
    \right)\\
    &=
    \frac{1}{2}\para{\beta_{k+1}^{-1}-\beta_k^{-1}}
    \norm{y-\prox_{\beta_k g}(y)}{2}^2\\
    &\leq \frac{1}{2}\para{\beta_{k+1}^{-1}-\beta_k^{-1}}
    \beta_k^2L_g^2,
    \end{aligned}
\end{equation*}
where we have used \propref{prop:moreauDisplacementBound} in the last inequality.
Applying this to the variation of $\Phi$ gives
\begin{equation*}
    \begin{aligned}
    \Phi_{k+1}(x)-\Phi_k(x)
    \leq
    \frac{1}{2}\para{\beta_{k+1}^{-1}-\beta_k^{-1}}\beta_k^2L_g^2\leq
    \frac{1}{2}\para{\beta_{k+1}^{-1}-\beta_k^{-1}}B_\prox^2,
    \end{aligned}
\end{equation*}
where the last inequality uses $\beta_{k+1}\leq \beta_k$  (Assumption~\ref{ass:beta}) and \propref{prop:idProxBound}.
\end{proof}

\begin{lemma}[Range bound for $\Phi_k$ over $\C$]\label{lem:lipschitzPhi}
Let Assumptions~\ref{ass:f}, \ref{ass:T}, \ref{ass:C}, \ref{ass:g}, \ref{ass:gamma}, and \ref{ass:beta} hold, let $\Phi_k$ be given by \eqref{e:phi} and let $B_{\prox}$ be given by \eqref{e:bprox}. Then, for all $x,y\in\C$ 
\begin{equation*}
(\forall k\in\N)\qquad
    \Phi_k(x)-\Phi_k(y)
    \leq
    \fconstant\dc+\beta_k^{-1}B_\prox \dtc.
\end{equation*}
\end{lemma}

\begin{proof}
For $x,y\in\C$, Section~\ref{sec:notation} provides 
\begin{equation}
    \label{e:36i}
f(x)-f(y)\leq\fconstant\dc.
\end{equation}
Suppose Assumption~\ref{ass:g}\ref{ass:g=indicator} holds. Since $\nabla\dist_{\D}^2(z)=2\para{z-\proj_{\D}(z)}$,
and for every $z\in T(\C)$, $\norm{z-\proj_{\D}(z)}{2}\leq e_{\D,T(\C)}$, the function $\dist_{\D}^2$ is $2e_{\D,T(\C)}$-Lipschitz on the convex set $T(\C)$. Hence
\begin{equation}
\label{e:36ii}
    \begin{aligned}
    g^{\beta_k}(Tx)-g^{\beta_k}(Ty)
    &=
    \frac{1}{2\beta_k}
    \para{\dist_{\D}^2(Tx)-\dist_{\D}^2(Ty)}\\
    &\leq
    \frac{e_{\D,T(\C)}}{\beta_k}\norm{Tx-Ty}{2}\\
    &\leq
    \beta_k^{-1}e_{\D,T(\C)}\dtc.
    \end{aligned}
\end{equation}
Suppose now that \assref{ass:g}\ref{ass:g=lipschitz} holds. By \propref{prop:moreauProperties}, for every $z\in R^m$, $\norm{\nabla g^{\beta_k}(z)}{2}\leq L_g$. Thus $g^{\beta_k}$ is $L_g$-Lipschitz on $\R^m$, and
\begin{equation}
\label{e:36iii}
    g^{\beta_k}(Tx)-g^{\beta_k}(Ty)
    \leq
    L_g\norm{Tx-Ty}{2}
    \leq
    L_g\dtc.
\end{equation}
Since Assumption~\ref{ass:beta} provides $\beta_k\leq\beta_0$, this implies $g^{\beta_k}(Tx)-g^{\beta_k}(Ty) \leq \beta_k^{-1}\beta_0L_g\dtc$.
Combining \eqref{e:36i}, \eqref{e:36ii}, and \eqref{e:36iii} gives the claim.
\end{proof}


\section{Frank-Wolfe Gaps and Stationarity Certificates}
\label{sec:gaps}

If $h$ is differentiable and the constraint set is $\C$, the \emph{Frank-Wolfe gap} is
\begin{equation}
\label{e:fwgap_vanilla}
    \max_{s\in\C}\ip{\nabla h(x),x-s}{}.
\end{equation}
For $x\in\C$, the quantity in \eqref{e:fwgap_vanilla} is nonnegative, and it vanishes at $x^*\in C$ if and only if
\begin{equation}
\label{e:gap-stationarity-rel}
    0\in \nabla h(x^*)+N_{\C}(x^*) \qquad\Leftrightarrow\qquad x^*\text{ is a stationary point of } \underset{x\in \C}{\min}\,h(x),
\end{equation}
where $N_{\C}(x^\star)$ is the normal cone of $\C$ at $x^*$. Thus, for \emph{smooth} nonconvex objective functions, the Frank-Wolfe gap is a natural first-order stationarity certificate \cite{CGsurvey}. However, the objective in \eqref{P} is generally \emph{nonsmooth}. This section delineates the smoothed gaps from the more meaningful stationarity certificates for the original problem \eqref{P}.

\subsection{The Smoothed Frank-Wolfe Gap}

For $\beta\in\left]0,\rho^{-1}\right[$ and $x\in\C$, we define the \emph{smoothed Frank-Wolfe gap} by
\begin{equation}
\label{eq:smoothedGap}
    \gap^\beta(x)
    :=
    \max_{s\in\C}
    \ip{\nabla f(x)+T^*\nabla g^\beta(Tx),x-s}{}.
\end{equation}
Equivalently, $\gap^\beta(x)$ is the usual Frank-Wolfe gap \eqref{e:fwgap_vanilla} with $h=\Phi_k$ given by \eqref{e:phi}.
In view of \eqref{e:gap-stationarity-rel}, $\gap^{\beta}$ characterizes stationarity for the smoothed problem $\min_{x\in\C}\Phi_{\beta}(x)$. Although Step~\ref{a:lmo} of \FRAMES\ allows one to compute $\gap^{\beta}(x_k)=\langle \nabla \Phi_k(x_k),\,x_k-s_k\rangle$,  
this is not a stationarity measure for \eqref{P} a priori.

\subsection{A Nonsmooth Frank-Wolfe Gap for the Indicator Case}

Suppose first that \assref{ass:g}\ref{ass:g=indicator}\ref{ass:g=indicator_nonempty} holds, so that $g=\iota_{\D}$ and there is a minimizer for \eqref{P}. Then $\widetilde{\C}:=\C\cap T^{-1}(\D)$ is nonempty, compact, and convex, and \eqref{P} is equivalent to minimizing $f$ over $\widetilde{\C}$.
The natural Frank-Wolfe stationarity certificate for this problem is below.
\begin{definition}[Signed nonsmooth Frank-Wolfe gap]\label{def:indicatorGap}
    Under \assref{ass:f}, \ref{ass:T}, \ref{ass:C}, and \ref{ass:g}\ref{ass:g=indicator}\ref{ass:g=indicator_nonempty}, the \emph{signed nonsmooth Frank-Wolfe gap} (or \emph{signed gap} for short) at $x\in \R^n$ is
    \begin{equation}
    \label{e:indigap-def}
        \widetilde{\gap}(x)
        :=
        \max_{s\in\widetilde{\C}}
        \ip{\nabla f(x),x-s}{}.
    \end{equation}
\end{definition}

\begin{remark}
    \label{r:infeas-indigap}
    Unlike the smoothed gap \eqref{eq:smoothedGap}, for $x\in\C$, $\widetilde{\gap}(x)$ need not be nonnegative. Although the iterates of \FRAMES\ always reside in $\C$, they are not guaranteed to satisfy $Tx_k\in\D$ for finite $k\in\N$. Thus, $\widetilde{\gap}$ should be interpreted as a signed stationarity measure.
\end{remark}

\begin{lemma}[Indicator stationarity certificate]
\label{lem:indicator_gap_certificate}
Suppose \assref{ass:g}\ref{ass:g=indicator}\ref{ass:g=indicator_nonempty} holds and define the signed gap as in \defref{def:indicatorGap}. If $\xbar\in\widetilde{\C}$, then $\xbar$ is a stationary point of $\min\limits_{x\in\widetilde{\C}}f(x)$ if and only if
\begin{equation*}
    \widetilde{\gap}(\xbar)=0
    \Leftrightarrow
    0\in\nabla f(\xbar)+N_{\widetilde{\C}}(\xbar).
\end{equation*}
Moreover, under \assref{ass:g}\ref{ass:g=indicator}\ref{ass:g=indicator_nonempty_ri}, this condition can be written
\begin{equation*}
    0\in\nabla f(\xbar)+N_{\C}(\xbar)+T^*N_{\D}(T\xbar).
\end{equation*}
\end{lemma}
\begin{proof}
The equivalence between $\widetilde{\gap}(x)=0$ and $0\in\nabla f(x)+N_{\widetilde{\C}}(x)$
when $x\in\widetilde{\C}$ is the standard Frank-Wolfe gap characterization \eqref{e:gap-stationarity-rel}.
The second claim follows from \cite[Proposition~6.19 \&  Theorem~16.47]{bauschke2017convex} which shows $N_{\widetilde{\C}}=N_{\C}+T^*\circ N_{\D}\circ T$.
\end{proof}

The indicator case therefore necessitates showing that two quantities vanish: feasibility, measured by $\dist_{\D}(Tx)$, and stationarity over the true feasible set $\widetilde{\C}$, measured by $\widetilde{\gap}(x)$.
We finish this section with a nonlinear generalization of  \cite[Lemma~3.5]{woodstock2025splitting} that provides a relationship between $\gap^{\beta}$, $\dist_{D}\circ T$, and $\widetilde{\gap}$.

\begin{lemma}[Indicator gap bound]
\label{lem:indicator_gap_transfer}
Let $\gap^{\beta}$ and $\widetilde{\gap}$ as in \eqref{eq:smoothedGap} and \eqref{e:indigap-def} respectively. If \assref{ass:g}\ref{ass:g=indicator}\ref{ass:g=indicator_nonempty} holds, then
\begin{equation}
\label{eq:indicator_gap_transfer}
(\forall x\in\C)(\forall \beta>0)\qquad
    \widetilde{\gap}(x)
    +
    \frac{1}{\beta}\dist_{\D}^2(Tx)
    \leq
    \gap^\beta(x).
\end{equation}
\end{lemma}

\begin{proof}
Since $g=\iota_{\D}$, Proposition~\ref{prop:moreauProperties} provides $
    \nabla (g^\beta\circ T)(x)
    =
    \frac{1}{\beta}T^*\para{Tx-\proj_{\D}(Tx)}$.
Thus,
\begin{equation}
    \begin{aligned}\label{eq:gap_transfer_step}
    -\gap^\beta(x)
    &=
    \min_{s\in\C}
    \Big\langle \nabla f(x)+\frac{1}{\beta}T^*\big(Tx-\proj_{\D}(Tx)\big),\,s-x\Big\rangle
    \\
    &\leq
    \min_{s\in\widetilde{\C}}
    \Big\langle\nabla f(x)+\frac{1}{\beta}T^*\big(Tx-\proj_{\D}(Tx)\big),\,s-x\Big\rangle\\
    &=
    \min_{s\in\widetilde{\C}}
    \big\langle \nabla f(x),\,s-x\big\rangle
    +
    \frac{1}{\beta}
    \langle Tx-\proj_{\D}(Tx),\,Ts-Tx\rangle.
    \end{aligned}
\end{equation}
For every $s\in\widetilde{\C}$, one has $Ts\in\D$. By \cite[Theorem~3.16]{bauschke2017convex},
$
    \ip{Tx-\proj_{\D}(Tx),Ts-\proj_{\D}(Tx)}{}\leq 0.
$
Consequently, for all $s\in\widetilde{\C}$,
    \begin{align}
    \nonumber
    \ip{Tx-\proj_{\D}(Tx),Ts-Tx}{}
    &=
    \ip{Tx-\proj_{\D}(Tx),Ts-\proj_{\D}(Tx)}{}
    -
    \norm{Tx-\proj_{\D}(Tx)}{2}^2\\
    \label{e:43i}
    &\leq
    -\dist_{\D}^2(Tx).
    \end{align}
 Combining  \eqref{eq:gap_transfer_step} and \eqref{e:43i} 
gives
$
    -\gap^\beta(x)
    \leq
    -\widetilde{\gap}(x)
    -
    \frac{1}{\beta}\dist_{\D}^2(Tx),
$
i.e., \eqref{eq:indicator_gap_transfer} holds.
\end{proof}

\subsection{A Nonsmooth Frank-Wolfe Gap for Lipschitz Weakly Convex \texorpdfstring{$g$}{g}}

We now suppose that \assref{ass:g}\ref{ass:g=lipschitz} holds, which implies that $g$ has full domain and is locally Lipschitz.
Since $g$ is weakly convex and defined on a finite-dimensional space, it is also \emph{subdifferentially regular} in the sense of \cite[Definition~7.25]{rockafellar1998variational}, since $g$ is the sum of two subdifferentially regular functions (the convex function \cite[Example~7.27]{rockafellar1998variational} $ h(y):=g(y)+\frac{\rho}{2}\norm{y}{2}^2$ and the twice continuously differentiable function $y\mapsto-\frac{\rho}{2}\norm{y}{2}^2$) satisfying the usual qualification conditions \cite[Corollary~10.9]{rockafellar1998variational}. Moreover, for locally Lipschitz subdifferentially regular functions, the regular, limiting, and Clarke subdifferentials are compatible in the sense of \cite[Definition~7.25, Corollary~8.11, Theorem~8.49]{rockafellar1998variational}. Thus, throughout this case, we write simply $\partial g$ for the Clarke subdifferential.

We first record the stationarity condition for \eqref{P} in this setting.

\begin{proposition}[Stationarity condition in the Lipschitz weakly convex case]
\label{prop:lipschitz_stationarity_condition}
Suppose \assref{ass:g}\ref{ass:g=lipschitz} holds. Then, for every $x\in\C$,
\begin{equation}
\label{eq:lipschitz_subdiff_formula}
    \partial\left(f+g\circ T+\iota_{\C}\right)(x)
    =
    \nabla f(x)+T^*\partial g(Tx)+N_{\C}(x),
\end{equation}
where $\partial g$ denotes the Clarke subdifferential. As a result, $x\in\C$ is a stationary point of \eqref{P} if and only if
\begin{equation}
\label{eq:lipschitz_stationarity_condition}
    0\in\nabla f(x)+T^*\partial g(Tx)+N_{\C}(x).
\end{equation}
Equivalently, $x\in\C$ is stationary if and only if there exists $\xi\in\partial g(Tx)$ such that
\begin{equation}
\label{eq:lipschitz_stationarity_condition_xi}
    -\para{\nabla f(x)+T^*\xi}\in N_{\C}(x).
\end{equation}
\end{proposition}

\begin{proof}
Since $f\in C^1$ on a neighborhood of $\C$, the smooth sum rule (e.g., \cite[Exercise~10.10]{rockafellar1998variational}) gives
\begin{equation}
\label{e:44i}
    \partial\left(f+g\circ T+\iota_{\C}\right)(x)
    =
    \nabla f(x)+\partial\left(g\circ T+\iota_{\C}\right)(x).
\end{equation}
 Since $\iota_{\C}$ is subdifferentially regular at every $x\in\C$ \cite[Example~7.27]{rockafellar1998variational}, and $g\circ T$ is subdifferentially regular because $g$ is subdifferentially regular and $T$ is linear, the sum rule for regular functions \cite[Corollary~10.9]{rockafellar1998variational} gives
\begin{equation}
\label{e:44ii}
    \partial\left(g\circ T+\iota_{\C}\right)(x)
    =
    \partial(g\circ T)(x)+N_{\C}(x).
\end{equation}
Finally, the Clarke chain rule for composition with a linear map gives
\begin{equation}
\label{e:44iii}
    \partial(g\circ T)(x)
    =
    T^*\partial g(Tx),
\end{equation}
see \cite[Theorem~2.3.10]{clarke1990optimization}. Combining \eqref{e:44i}, \eqref{e:44ii}, and \eqref{e:44iii} yields \eqref{eq:lipschitz_subdiff_formula}. The stationarity condition \eqref{eq:lipschitz_stationarity_condition} follows from the definition of a stationary point, and \eqref{eq:lipschitz_stationarity_condition_xi} is the same inclusion written with a selected subgradient $\xi\in\partial g(Tx)$.
\end{proof}

Proposition~\ref{prop:lipschitz_stationarity_condition} motivates the following nonsmooth notion of a Frank-Wolfe gap.
\begin{definition}[Nonsmooth Frank-Wolfe subgradient gap]\label{def:nonsmoothGap}
    Under \assref{ass:f}, \ref{ass:T}, \ref{ass:C}, and \ref{ass:g}\ref{ass:g=lipschitz}, we define the \emph{nonsmooth Frank-Wolfe subgradient gap} (or \emph{subgradient gap} for short) at $x\in\R^n$ with subgradient $\xi\in\partial g(Tx)$ to be
    \begin{equation*}
        \gap(x;\xi)
        :=
        \max_{s\in\C}
        \ip{\nabla f(x)+T^*\xi,x-s}{}.
    \end{equation*}
\end{definition}

\begin{lemma}[Lipschitz weakly convex stationarity certificate]
\label{lem:lipschitz_gap_certificate}
Suppose \assref{ass:g}\ref{ass:g=lipschitz} holds and define the subgradient gap as in \defref{def:nonsmoothGap}. A point $x\in\C$ satisfies the stationarity condition
\begin{equation}
\label{e:ns-stationary-asm2}
    0\in\nabla f(x)+T^*\partial g(Tx)+N_{\C}(x)
\end{equation}
if and only if there exists $\xi\in\partial g(Tx)$ such that
\begin{equation*}
    \gap(x;\xi)=0.
\end{equation*}
\end{lemma}

\begin{proof}
Fix $x\in\C$ and $\xi\in\partial g(Tx)$.
Since $x\in\C$, the choice $s=x$ gives $\gap(x;\xi)\geq0$. Moreover, $\gap(x;\xi)=0$ $\Leftrightarrow$ $\forall s\in\C\;\;
    \ip{\nabla f(x)+T^*\xi,x-s}{}\leq0$ $\Leftrightarrow$ 
    $\forall s\in\C\;\;
    \ip{-(\nabla f(x)+T^*\xi),s-x}{}\leq0
    $,
which is precisely the condition $
    -(\nabla f(x)+T^*\xi)\in N_{\C}(x)$.
Thus, $\gap(x;\xi)=0$ for some $\xi\in\partial g(Tx)$ if and only if \eqref{e:ns-stationary-asm2} holds.
\end{proof}

The smoothed gap and the nonsmooth gap are related but not identical. Indeed, the Moreau envelope produces the vector
\begin{equation*}
    \xi_\beta(x)
    :=
    \nabla g^\beta(Tx)
    =
    \frac{Tx-\prox_{\beta g}(Tx)}{\beta}.
\end{equation*}
By the proximal optimality condition, $\xi_\beta(x)\in\partial g(\prox_{\beta g}(Tx))$.
Thus, $\xi_\beta(x)$ is a Clarke subgradient of $g$ at $\prox_{\beta g}(Tx)$, not necessarily at $Tx$. Due to this, $\gap^\beta(x)$ is not simply the subgradient gap, $\gap(x;\xi)$, evaluated at $\xi=\xi_\beta(x)$. This mismatch is one of the main motivators for the gap-transfer analysis in Section~\ref{sec:gap_transfer}.

\begin{remark}[The two nonsmooth gaps are distinct]
\label{rmk:gaps-are-distinct}
The signed gap given in \defref{def:indicatorGap} and the subgradient gap given in \defref{def:nonsmoothGap} are different stationarity certificates. Using \defref{def:nonsmoothGap} under Assumption~\ref{ass:g}\ref{ass:g=indicator} would require an element of $\partial\iota_{\D}(Tx)=N_{\D}(Tx)$, which is empty whenever $Tx\notin\D$. In view of Remark~\ref{r:infeas-indigap}, since \FRAMES\ may generate iterates satisfying $x_k\in\C$ but $Tx_k\notin\D$, the subgradient gap is not informative, as it would be identically $+\infty$. On the other hand, \defref{def:indicatorGap} provides a finite signed quantity that can be analyzed in tandem with the feasibility measure $\dist_{\D}(Tx_k)$.
\end{remark}


\section{Smoothed Convergence Rates and Asymptotic Stationarity}
\label{sec:smoothed_convergence}

In this section, we prove that the average and best smoothed gaps generated by \FRAMES\ converge to zero at an explicit rate. The main difference from the standard smooth Frank-Wolfe analysis is that the objective changes with $k$, since the smoothing parameter $\beta_k$ is decreasing to zero. The estimates from \secref{sec:preliminary} allow us to control both the usual Frank-Wolfe descent term and the variation of the smoothed objectives $\Phi_k$.

\subsection{An Energy Estimate for the Smoothed Gaps}

We begin with the basic descent estimate for one iteration of \FRAMES, using the definition of smoothed Frank-Wolfe gap from \eqref{eq:smoothedGap} throughout.

\begin{lemma}[One-step smoothed descent]
\label{lem:one_step_smoothed_descent}
Let Assumptions~\ref{ass:f}, \ref{ass:T}, \ref{ass:C}, \ref{ass:g}, \ref{ass:gamma}, and \ref{ass:beta} hold and let $\seq{x_k}$ be generated by \FRAMES. Then, for every $k\in\N$,
\begin{equation}
\label{e:phi_k-descent}
    \Phi_k(x_{k+1})
    \leq
    \Phi_k(x_k)
    -
    \gamma_k\gap^{\beta_k}(x_k)
    +
    \frac{L_k\gamma_k^2}{2}\dc^2,
\end{equation}
where $L_k$ is the estimated smoothness constant defined in \eqref{eq:Lk}. Moreover,
\begin{equation}
\label{e:phi-descent}
    \Phi_{k+1}(x_{k+1})
    \leq
    \Phi_k(x_k)
    -
    \gamma_k\gap^{\beta_k}(x_k)
    +
    \frac{L_k\gamma_k^2}{2}\dc^2
    +
    \frac{1}{2}\para{\beta_{k+1}^{-1}-\beta_k^{-1}}B_\prox^2.
\end{equation}
where $B_\prox$ is defined in \eqref{e:bprox}.
\end{lemma}
\begin{proof}
By \propref{prop:phiDescent}, applied to $\Phi_k$ at $x_k$ and $x_{k+1}$,
\begin{equation*}
    \Phi_k(x_{k+1})
    \leq
    \Phi_k(x_k)
    +
    \ip{\nabla\Phi_k(x_k),x_{k+1}-x_k}{}
    +
    \frac{L_k}{2}\norm{x_{k+1}-x_k}{2}^2.
\end{equation*}
Since $x_{k+1}=x_k+\gamma_k(s_k-x_k)$, we have
\begin{equation*}
    \ip{\nabla\Phi_k(x_k),x_{k+1}-x_k}{}
    =
    -\gamma_k\ip{\nabla\Phi_k(x_k),x_k-s_k}{}
    =
    -\gamma_k\gap^{\beta_k}(x_k),
\end{equation*}
where the last equality follows from the definition of $s_k$ in \FRAMES. Moreover, since $x_k,s_k\in\C$,
\begin{equation*}
    \norm{x_{k+1}-x_k}{2}
    =
    \gamma_k\norm{s_k-x_k}{2}
    \leq
    \gamma_k\dc.
\end{equation*}
This proves \eqref{e:phi_k-descent}. Adding \eqref{e:lem-var-phi} of Lemma~\ref{lem:phiEstimate} to \eqref{e:phi_k-descent} yields \eqref{e:phi-descent}.
\end{proof}

Summing Lemma~\ref{lem:one_step_smoothed_descent} over $k\in\N$ yields the main estimate used to find convergence rates.

\begin{lemma}[Energy estimate for smoothed gaps]
\label{lem:energy}
Let Assumptions~\ref{ass:f}, \ref{ass:T}, \ref{ass:C}, \ref{ass:g}, \ref{ass:gamma}, and \ref{ass:beta} hold and let $\seq{x_k}$ be generated by \FRAMES. Then, for every $N\in\N^*$,
\begin{equation*}
    0
    \leq
    \sum_{k=0}^{N-1}\gap^{\beta_k}(x_k)
    \leq
    \frac{\fconstant\dc+B\beta_N^{-1}}{\gamma_{N-1}}
    +
    \frac{\dc^2}{2}\sum_{k=0}^{N-1}\gamma_k\para{\gradfconstant
    +
    \norm{T}{\op}^2M_{\rho,\beta_0}\beta_k^{-1}},
\end{equation*}
where $B:=\max\{\frac{1}{2}B_\prox^2,\ B_\prox\dtc\}$ and $M_{\rho,\beta_0}:=\max\{1,\tfrac{\rho\beta_0}{1-\rho\beta_0}\}$.
\end{lemma}

\begin{proof}
For each $k\in\N$, define $\Phi_k^\star:=\min_{x\in\C}\Phi_k(x)$ and $r_k:=\Phi_k(x_k)-\Phi_k^\star$.
By \lemref{lem:one_step_smoothed_descent}, for every $k\in\N$,
\begin{equation*}
    \gamma_k\gap^{\beta_k}(x_k)
    \leq
    \Phi_k(x_k)-\Phi_{k+1}(x_{k+1})
    +
    \frac{1}{2}\para{\beta_{k+1}^{-1}-\beta_k^{-1}}B_\prox^2
    +
    \frac{L_k\gamma_k^2}{2}\dc^2.
\end{equation*}
Since $\beta_{k+1}\leq\beta_k$, the Moreau envelopes are pointwise nondecreasing in $k$, and hence $\forall x\in\R^n$ $\Phi_k(x)\leq \Phi_{k+1}(x)$. Therefore $\Phi_k^\star\leq\Phi_{k+1}^\star$, and so $\Phi_k(x_k)-\Phi_{k+1}(x_{k+1}) \leq r_k-r_{k+1}$.
Thus,
\begin{equation*}
    \gamma_k\gap^{\beta_k}(x_k)
    \leq
    r_k-r_{k+1}
    +
    \frac{1}{2}\para{\beta_{k+1}^{-1}-\beta_k^{-1}}B_\prox^2
    +
    \frac{L_k\gamma_k^2}{2}\dc^2.
\end{equation*}
Dividing by $\gamma_k$ and summing from $k=0$ to $N-1$ gives
\begin{equation*}
    \begin{aligned}
    \sum_{k=0}^{N-1}\gap^{\beta_k}(x_k)
    &\leq
    \sum_{k=0}^{N-1}\frac{r_k-r_{k+1}}{\gamma_k}
    +
    \frac{1}{2}B_\prox^2\sum_{k=0}^{N-1}
    \frac{\beta_{k+1}^{-1}-\beta_k^{-1}}{\gamma_k}
    +
    \frac{\dc^2}{2}\sum_{k=0}^{N-1}\gamma_kL_k.
    \end{aligned}
\end{equation*}

We now bound the first two sums. By summation by parts,
\begin{equation*}
    \sum_{k=0}^{N-1}\frac{r_k-r_{k+1}}{\gamma_k}
    =
    \frac{r_0}{\gamma_0}
    +
    \sum_{k=1}^{N-1}
    \para{\frac{1}{\gamma_k}-\frac{1}{\gamma_{k-1}}}r_k
    -
    \frac{r_N}{\gamma_{N-1}}.
\end{equation*}
Since $r_N\geq0$ and $\seq{\gamma_k}$ is nonincreasing, we get
\begin{equation*}
    \sum_{k=0}^{N-1}\frac{r_k-r_{k+1}}{\gamma_k}
    \leq
    \frac{r_0}{\gamma_0}
    +
    \sum_{k=1}^{N-1}
    \para{\frac{1}{\gamma_k}-\frac{1}{\gamma_{k-1}}}r_k.
\end{equation*}
By \lemref{lem:lipschitzPhi}, for every $k\in\N$, $r_k = \Phi_k(x_k)-\Phi_k^\star \leq \fconstant\dc+\beta_k^{-1}B_\prox\dtc$.
Using 
$
    B:=\max\{\frac{1}{2}B_\prox^2,\ B_\prox\dtc\},
$
we obtain
\begin{equation*}
    \begin{aligned}
    \sum_{k=0}^{N-1}\gap^{\beta_k}(x_k)
    &\leq
    \frac{\fconstant\dc+B\beta_0^{-1}}{\gamma_0}
    +
    \sum_{k=1}^{N-1}
    \para{\frac{1}{\gamma_k}-\frac{1}{\gamma_{k-1}}}
    \para{\fconstant\dc+B\beta_k^{-1}}\\
    &\quad+
    B\sum_{k=0}^{N-1}
    \frac{\beta_{k+1}^{-1}-\beta_k^{-1}}{\gamma_k}
    +
    \frac{\dc^2}{2}\sum_{k=0}^{N-1}\gamma_kL_k.
    \end{aligned}
\end{equation*}
The terms involving $\fconstant\dc$ telescope:
\begin{equation*}
    \frac{\fconstant\dc}{\gamma_0}
    +
    \sum_{k=1}^{N-1}
    \para{\frac{1}{\gamma_k}-\frac{1}{\gamma_{k-1}}}\fconstant\dc
    =
    \frac{\fconstant\dc}{\gamma_{N-1}}.
\end{equation*}
The terms involving $B$ also telescope, giving
\begin{equation*}
    \begin{aligned}
    &\frac{B\beta_0^{-1}}{\gamma_0}
    +
    B\sum_{k=1}^{N-1}
    \para{\frac{1}{\gamma_k}-\frac{1}{\gamma_{k-1}}}\beta_k^{-1}
    +
    B\sum_{k=0}^{N-1}
    \frac{\beta_{k+1}^{-1}-\beta_k^{-1}}{\gamma_k}=
    \frac{B\beta_N^{-1}}{\gamma_{N-1}}.
    \end{aligned}
\end{equation*}
Combining these estimates yields
\begin{equation}
\label{e:gap-ubound}
    \sum_{k=0}^{N-1}\gap^{\beta_k}(x_k)
    \leq
    \frac{\fconstant\dc+B\beta_N^{-1}}{\gamma_{N-1}}
    +
    \frac{\dc^2}{2}\sum_{k=0}^{N-1}\gamma_kL_k.
\end{equation}
Since $\beta_k\leq\beta_0$, we have  $\frac{\rho}{1-\rho\beta_k}
    \leq
    \frac{\rho\beta_0}{1-\rho\beta_0}\beta_k^{-1}$,
and therefore
$
    \max\left\{\beta_k^{-1},\frac{\rho}{1-\rho\beta_k}\right\}
    \leq
    M_{\rho,\beta_0}\beta_k^{-1}.
$ 
Consequently,
$
    L_k
    \leq
    \gradfconstant
    +
    \norm{T}{\op}^2M_{\rho,\beta_0}\beta_k^{-1};
$ with \eqref{e:gap-ubound}, this shows the upper bound.
The lower bound follows from nonnegativity of the smoothed gap and the fact that $x_k\in \C$.
\end{proof}

\subsection{Convergence Rate for the Smoothed Gaps}

We now specialize the energy estimate to specific open-loop schedules in order to derive rates.

\begin{theorem}[smoothed gap rates for power schedules]
\label{thm:rates}
Fix $p,q\in\left]0,1\right[$ satisfying $q<p$ and $p+q<1$ and set $\gamma_k=(k+1)^{-p}$ and $\beta_k=\beta_0(k+1)^{-q}$ with $\beta_0\in\left]0,\rho^{-1}\right[$. Let Assumptions~\ref{ass:f}, \ref{ass:T}, \ref{ass:C}, and \ref{ass:g} hold and let $\seq{x_k}$ be generated by \FRAMES. Then, there exists some $C_{p,q}\geq 0$, defined in \eqref{eq:Cpq_constant}, such that, for every $N\in\N^*$,
\begin{equation}
\label{eq:pq_smoothed_average_rate}
    0
    \leq
    \frac{1}{N}\sum_{k=0}^{N-1}\gap^{\beta_k}(x_k)
    \leq
    C_{p,q}N^{-\min\{p-q,\ 1-p-q\}}.
\end{equation}
Consequently,
\begin{equation}
\label{eq:pq_smoothed_min_rate}
    0
    \leq
    \min_{0\leq k\leq N-1}\gap^{\beta_k}(x_k)
    \leq
    C_{p,q}N^{-\min\{p-q,\ 1-p-q\}}.
\end{equation}
In particular, for $p=1/2$ and $q=1/4$ one has
$\min\{p-q,\ 1-p-q\}=1/4$ and a constant $C\geq0$, defined in \eqref{eq:C_constant}, independent of $N$ such that
\begin{equation*}
    0
    \leq
    \frac{1}{N}\sum_{k=0}^{N-1}\gap^{\beta_k}(x_k)
    \leq
    CN^{-1/4}\qquad\text{and}
    \qquad
    0
    \leq
    \min_{0\leq k\leq N-1}\gap^{\beta_k}(x_k)
    \leq
    CN^{-1/4}.
\end{equation*}
\end{theorem}

\begin{proof}
The chosen schedules satisfy Assumptions~\ref{ass:gamma} and~\ref{ass:beta}. By \lemref{lem:energy}, we have
\begin{equation}
\label{eq:pq_energy_start}
 0\leq   \sum_{k=0}^{N-1}\gap^{\beta_k}(x_k)
    \leq
    \frac{\fconstant\dc+B\beta_N^{-1}}{\gamma_{N-1}}
    +
    \frac{\dc^2}{2}
    \sum_{k=0}^{N-1}
    \gamma_k
    \para{
        \gradfconstant
        +
        \norm{T}{\op}^2M_{\rho,\beta_0}\beta_k^{-1}
    },
\end{equation}
where $B:=\max\{\frac{1}{2}B_\prox^2,\ B_\prox\dtc\}$.
For the power schedules, $\gamma_{N-1}^{-1}=N^p$ and $\beta_N^{-1} = \beta_0^{-1}(N+1)^q$. Since $N\geq1$ and $q\in\left]0,1\right[$, we have $(N+1)^q\leq 2^qN^q$
and $N^{p-1}\leq N^{p+q-1}$.
Therefore
\begin{equation}
\label{eq:pq_range_bound}
    \begin{aligned}
    \frac{\fconstant\dc+B\beta_N^{-1}}{N\gamma_{N-1}}
    &=
    \fconstant\dc N^{p-1}
    +
    B\beta_0^{-1}N^{p-1}(N+1)^q                                      \\
    &\leq
    \fconstant\dc N^{p+q-1}
    +
    2^qB\beta_0^{-1}N^{p+q-1}                                           \\
    &=
    \para{\fconstant\dc+2^qB\beta_0^{-1}}N^{-(1-p-q)}.
    \end{aligned}
\end{equation}
It remains to consider the second term in the right hand side of \eqref{eq:pq_energy_start}. Since $\gamma_k=(k+1)^{-p}$ and 
$\beta_k^{-1}=\beta_0^{-1}(k+1)^q$ we obtain that $\frac{1}{N}
    \sum_{k=0}^{N-1}
    \gamma_k
(
        \gradfconstant
        +
        \norm{T}{\op}^2M_{\rho,\beta_0}\beta_k^{-1}
  )$ is equal to
\begin{equation}
\label{eq:pq_curvature_start}
    \frac{\gradfconstant}{N}
    \sum_{k=0}^{N-1}(k+1)^{-p}
    +
    \frac{\norm{T}{\op}^2M_{\rho,\beta_0}\beta_0^{-1}}{N}
    \sum_{k=0}^{N-1}(k+1)^{-(p-q)}.
\end{equation}
Since $p\in\left]0,1\right[$ and $p-q\in\left]0,1\right[$, the integral test gives
\begin{equation}
\label{eq:pq_power_sum_p}
    \sum_{k=0}^{N-1}(k+1)^{-p}
    =
    \sum_{j=1}^{N}j^{-p}
    \leq
    \frac{N^{1-p}}{1-p}
    \;\text{and}\;
     \sum_{k=0}^{N-1}(k+1)^{-(p-q)}
    =
    \sum_{j=1}^{N}j^{-(p-q)}
    \leq
    \frac{N^{1-p+q}}{1-p+q}.
\end{equation}
Combining the expression \eqref{eq:pq_curvature_start} with \eqref{eq:pq_power_sum_p} and the fact $N^{-p}\leq N^{-(p-q)}$ yields
\begin{equation}
\label{eq:pq_curvature_bound}
    \begin{aligned}
        \frac{1}{N}
    \sum_{k=0}^{N-1}
    \gamma_k
    \para{
        \gradfconstant
        +
        \norm{T}{\op}^2M_{\rho,\beta_0}\beta_k^{-1}
    }
    &\leq
        \frac{\gradfconstant}{1-p}N^{-p}
        +
        \frac{\norm{T}{\op}^2M_{\rho,\beta_0}\beta_0^{-1}}{1-p+q}
        N^{-(p-q)}
\\
    &\leq
   \left(
        \frac{\gradfconstant}{1-p}
        +
        \frac{\norm{T}{\op}^2M_{\rho,\beta_0}\beta_0^{-1}}{1-p+q}
   \right)
    N^{-(p-q)}.
    \end{aligned}
\end{equation}

Dividing \eqref{eq:pq_energy_start} by $N$ and using
\eqref{eq:pq_range_bound} and \eqref{eq:pq_curvature_bound}, we get
\begin{equation*}
\begin{split}
   0\leq \frac{1}{N}
    \sum_{k=0}^{N-1}\gap^{\beta_k}(x_k)
    \leq
    \para{\fconstant\dc+2^qB\beta_0^{-1}}N^{-(1-p-q)}
    +\\
    \frac{\dc^2}{2}
    \left(
        \frac{\gradfconstant}{1-p}
        +
        \frac{\norm{T}{\op}^2M_{\rho,\beta_0}\beta_0^{-1}}{1-p+q}
    \right)
    N^{-(p-q)}.
    \end{split}
\end{equation*}
Hence \eqref{eq:pq_smoothed_average_rate} holds,
where $B:=\max\{\frac{1}{2}B_\prox^2,\ B_\prox\dtc\}$, $M_{\rho,\beta_0}=\max\{1,\ \frac{\rho\beta_0}{1-\rho\beta_0}\}$, and
\begin{equation}
    \label{eq:Cpq_constant}
    C_{p,q}:= \fconstant\dc + 2^qB\beta_0^{-1} + \frac{\dc^2}{2}
    \left(
        \frac{\gradfconstant}{1-p}
        +
        \frac{\norm{T}{\op}^2M_{\rho,\beta_0}\beta_0^{-1}}{1-p+q}
    \right).
\end{equation}
The best-gap estimate \eqref{eq:pq_smoothed_min_rate} follows from nonnegativity of the smoothed gaps
\begin{equation*}
   0\leq \min_{0\leq k\leq N-1}\gap^{\beta_k}(x_k)
    \leq
    \frac{1}{N}
    \sum_{k=0}^{N-1}\gap^{\beta_k}(x_k).
\end{equation*}
Finally, substituting $p=1/2$ and $q=1/4$ gives the claimed $N^{-1/4}$ bounds with 
\begin{equation}
    \label{eq:C_constant}
    C := \fconstant\dc
    +
    2^{1/4}B\beta_0^{-1}
    +
    \gradfconstant\dc^2
    +
    \frac{2\dc^2}{3}\norm{T}{\op}^2M_{\rho,\beta_0}\beta_0^{-1}.
\end{equation}
\end{proof}

\begin{remark}[Smoothed gaps versus nonsmooth certificates]
\label{rmk:smoothed_gap_schedule_choice}
For fixed $q$, the smoothed gap exponent in \thmref{thm:rates} is maximized by balancing $p-q=1-p-q$, which gives $p=1/2$ and the smoothed gap exponent $1/2-q$. Thus, if one only tries to optimize the smoothed gaps, it is tempting to choose $q$ very small, obtaining a smoothed gap rate close to $N^{-1/2}$. This perspective is natural in smoothing-based analyses, as was done in \cite{woodstock2025splitting,halbey2025efficient}.

However, the smoothed gap does not directly certify stationarity for \eqref{P}. In order to transfer the convergence rate of the smoothed gaps to determine a convergence rate for the nonsmooth gaps (see \secref{sec:gaps}), it turns out the nonsmooth convergence rate also depends on the smoothing error $\beta_k=\mathcal{O}(k^{-q})$, as shown later in \secref{sec:gap_transfer}. Hence a schedule that is favorable for the smoothed gaps alone can smooth too slowly for the final nonsmooth stationarity certificate. Later, Remark~\ref{rmk:optimal_power_balance_nonsmooth} demonstrates $p=1/2$ and $q=1/4$ to be optimal.
\end{remark}

\begin{corollary}[Last-half best iterate]
\label{cor:last_half_best_iterate}
Under the same assumptions as \thmref{thm:rates}, let $N\geq2$, $p=1/2$, $q=1/4$, and choose
\begin{equation*}
    \kmin\in
    \argmin_{\lfloor N/2\rfloor\leq k\leq N-1}
    \gap^{\beta_k}(x_k).
\end{equation*}
Then,
\begin{equation*}
    \gap^{\beta_{\kmin}}(x_{\kmin})
    \leq
    2CN^{-1/4}\quad\mbox{and}\quad
    \beta_{\kmin}
    \leq
    2^{1/4}\beta_0N^{-1/4}.
\end{equation*}
\end{corollary}

\begin{proof}
Since the set $\{\lfloor N/2\rfloor,\ldots,N-1\}$ has cardinality at least $N/2$, we have
\begin{equation*}
    \begin{aligned}
    \gap^{\beta_{\kmin}}(x_{\kmin})
    &\leq
    \frac{2}{N}
    \sum_{k=\lfloor N/2\rfloor}^{N-1}
    \gap^{\beta_k}(x_k)\leq
    \frac{2}{N}
    \sum_{k=0}^{N-1}
    \gap^{\beta_k}(x_k)
    \leq
    2CN^{-1/4}.
    \end{aligned}
\end{equation*}
Moreover, since $\kmin\geq\lfloor N/2\rfloor$,
\begin{equation*}
    \beta_{\kmin}
    =
    \beta_0(\kmin+1)^{-1/4}
    \leq
    \beta_0\para{\lfloor N/2\rfloor+1}^{-1/4}
    \leq
    2^{1/4}\beta_0N^{-1/4}.
\end{equation*}
\end{proof}

\subsection{Subsequential Convergence to Stationary Points of \texorpdfstring{\eqref{P}}{(P)}}
\begin{corollary}[Smoothed gap-vanishing subsequences]
\label{cor:gap_vanishing_subsequence}
Under the assumptions of \thmref{thm:rates}, there exists an increasing sequence $\seqj{k_j}\subset\N$ such that
\begin{equation*}
    x_{k_j}\to\xbar\in\C\quad\text{and}\quad\gap^{\beta_{k_j}}(x_{k_j})\to0.
\end{equation*}
\end{corollary}

\begin{proof}
Since the smoothed gaps are nonnegative and their average converges to zero by \thmref{thm:rates}, we have
\begin{equation*}
    \liminf_{k\to+\infty}\gap^{\beta_k}(x_k)=0.
\end{equation*}
Thus, there exists an increasing sequence $\seqj{k_j}$ such that
\begin{equation*}
    \gap^{\beta_{k_j}}(x_{k_j})\to0.
\end{equation*}
The compactness of $\C$ and the fact that $x_k\in\C$ for every $k$ imply that $\seqj{x_{k_j}}$ has a convergent subsequence; picking this subsequence and relabeling indices gives the final result.
\end{proof}

\thmref{thm:rates} shows that \FRAMES\ controls the computable smoothed certificates $\gap^{\beta_k}(x_k)$ with few assumptions. The convergence results obtained are in fact enough to ensure stationarity for \eqref{P} asymptotically, which we show next, starting with the indicator case.

\begin{theorem}[Indicator subsequential stationarity]
\label{thm:indicator_subsequential_stationarity}
Suppose \assref{ass:f}, \ref{ass:T}, \ref{ass:C}, and \ref{ass:g}\ref{ass:g=indicator}\ref{ass:g=indicator_nonempty} hold and suppose that $\seqj{u_j}\subset\C$ with $\seqj{\alpha_j}\subset]0,+\infty[$ define a convergent smoothed gap-vanishing sequence, i.e.,
\begin{equation*}
    \alpha_j\to 0,\qquad u_{j}\to\ubar,
    \qquad\text{and}\qquad
    \gap^{\alpha_{j}}(u_{j})\to0.
\end{equation*}
Then $\ubar\in\widetilde{\C}$ and $\widetilde{\gap}(\ubar)=0$, using the definition of signed gap given in \defref{def:indicatorGap}.
Equivalently, $0\in\nabla f(\ubar)+N_{\widetilde{\C}}(\ubar)$.
Under \assref{ass:g}\ref{ass:g=indicator}\ref{ass:g=indicator_nonempty_ri}, this can be rewritten as
\begin{equation*}
    0\in\nabla f(\ubar)+N_{\C}(\ubar)+T^*N_{\D}(T\ubar).
\end{equation*}
\end{theorem}

\begin{proof}
Since $u_{j}\in\C$ for every $j\in\N^*$ and $\C$ is compact, the limit satisfies $\ubar\in\C$.

We first prove feasibility. By \lemref{lem:indicator_gap_transfer},
\begin{equation*}
    \widetilde{\gap}(u_{j})
    +
    \alpha_{j}^{-1}\dist_{\D}^2(Tu_{j})
    \leq
    \gap^{\alpha_{j}}(u_{j}).
\end{equation*}
For every $u\in\C$ and every $s\in\widetilde{\C}$,
$
   \widetilde{\gap}(u)\geq \ip{\nabla f(u),u-s}{}
    \geq
    -\fconstant\dc.
$ and therefore
\begin{equation*}
    \alpha_{j}^{-1}\dist_{\D}^2(Tu_{j})
    \leq
    \gap^{\alpha_{j}}(u_{j})
    +
    \fconstant\dc.
\end{equation*}
Since $\alpha_{j}\to0$ and $\gap^{\alpha_{j}}(u_{j})\to0$, it follows that $\dist_{\D}(Tu_{j})\to0$.
By continuity of $u\mapsto\dist_{\D}(Tu)$, we get $\dist_{\D}(T\ubar)=0$. Thus $T\ubar\in\D$, and so $\ubar\in\widetilde{\C}$.

We now prove stationarity over $\widetilde{\C}$. From \lemref{lem:indicator_gap_transfer}, $\widetilde{\gap}(u_{j}) \leq \gap^{\alpha_{j}}(u_{j})$.
Taking the limit superior gives $\limsup_{j\to+\infty}\widetilde{\gap}(u_{j})\leq0$.
Since $\widetilde{\C}$ is compact and $\nabla f$ is continuous, the function $\widetilde{\gap}$ is continuous. Hence $\widetilde{\gap}(\ubar)\leq0$.
But $\ubar\in\widetilde{\C}$, so choosing $s=\ubar$ in the definition of $\widetilde{\gap}$ gives $\widetilde{\gap}(\ubar)\geq0$. Therefore $\widetilde{\gap}(\ubar)=0$. The equivalence with the normal cone stationarity condition follows from \lemref{lem:indicator_gap_certificate}.
\end{proof}

\begin{theorem}[Lipschitz weakly convex subsequential stationarity]
\label{thm:lipschitz_subsequential_stationarity}
Suppose \assref{ass:f}, \ref{ass:T}, \ref{ass:C}, and \ref{ass:g}\ref{ass:g=lipschitz} hold and suppose that $\seqj{u_j}\subset\C$ with $\seqj{\alpha_j}\subset]0,\rho^{-1}[$ define a convergent smoothed gap-vanishing sequence, i.e.,
\begin{equation*}
    \alpha_j\to0,
    \qquad u_{j}\to\ubar,
    \qquad\text{and}\qquad
    \gap^{\alpha_{j}}(u_{j})\to0.
\end{equation*}
Then, using the definition of subgradient gap in \defref{def:nonsmoothGap}, there exists $\bar{\xi}\in\partial g(T\ubar)$ such that $\gap(\ubar,\bar{\xi})=0$ and thus
\begin{equation*}
    0\in\nabla f(\ubar)+T^*\bar{\xi}+N_{\C}(\ubar),\quad\mbox{i.e.,}\quad 0\in\nabla f(\ubar)+T^*\partial g(T\ubar)+N_{\C}(\ubar).
\end{equation*}
\end{theorem}

\begin{proof}
For each $j\in\N$, define $\xi_j:=\nabla g^{\alpha_{j}}(Tu_{j})$. By \propref{prop:moreauDisplacementBound},
\begin{equation}
\label{e:lip-grad0}
    \norm{Tu_{j}-\prox_{\alpha_{j}g}(Tu_{j})}{2}
    \leq
    \alpha_{j}L_g
    \to0.
\end{equation}
Since $u_{j}\to\ubar$ and $T$ is linear, this implies $\prox_{\alpha_{j}g}(Tu_{j})\to T\ubar$. Moreover, \propref{prop:moreauProperties} gives
\begin{equation*}
    \norm{\xi_j}{2}\leq L_g
    \qquad
    \forall j\in\N.
\end{equation*}
Hence, after passing to a subsequence if necessary, there exists $\bar{\xi}\in\R^m$ such that $\xi_j\to\bar{\xi}$. For every $j$, the proximal optimality condition gives
\begin{equation*}
    \xi_j\in\partial g(\prox_{\alpha_{j}g}(Tu_{j})).
\end{equation*}
Since \eqref{e:lip-grad0} yields $\prox_{\alpha_{j}g}(Tu_{j})\to T\ubar$, $\xi_j\to\bar{\xi}$, and since the Clarke subdifferential is outer semicontinuous for locally Lipschitz functions \cite[Proposition~2.1.5]{clarke1990optimization}, we obtain $\bar{\xi}\in\partial g(T\ubar)$.

It remains to show the normal cone condition. For every $s\in\C$, \eqref{eq:smoothedGap} gives
\begin{equation*}
    \ip{\nabla f(u_{j})+T^*\xi_j,u_{j}-s}{}
    \leq
    \gap^{\alpha_{j}}(u_{j}).
\end{equation*}
Passing to the limit along the chosen subsequence yields
\begin{equation*}
    \ip{\nabla f(\ubar)+T^*\bar{\xi},\ubar-s}{}\leq0
    \qquad
    \forall s\in\C.
\end{equation*}
Equivalently,
$-\para{\nabla f(\ubar)+T^*\bar{\xi}}\in N_{\C}(\ubar)$. Thus $0\in\nabla f(\ubar)+T^*\bar{\xi}+N_{\C}(\ubar)$ as claimed.
\end{proof}

Combining \corref{cor:gap_vanishing_subsequence} with \thmref{thm:indicator_subsequential_stationarity} and \thmref{thm:lipschitz_subsequential_stationarity} gives the following.

\begin{corollary}[Existence of stationary cluster points]
\label{cor:stationary_cluster_points}
Under the assumptions of \thmref{thm:rates} with \assref{ass:g}\ref{ass:g=indicator}\ref{ass:g=indicator_nonempty} if $g=\iota_{\D}$, the sequence $\seq{x_k}$ generated by \FRAMES\ admits a convergent smoothed gap-vanishing subsequence with limit point $\xbar\in\C$ that is stationary for \eqref{P}, i.e.,
\begin{equation*}
    \begin{cases}
        0\in\nabla f(\xbar)+N_{\widetilde{\C}}(\xbar) & \text{under \assref{ass:g}\ref{ass:g=indicator}\ref{ass:g=indicator_nonempty}},\\
        0\in\nabla f(\xbar)+T^*\partial g(T\xbar)+N_{\C}(\xbar) & \text{under \assref{ass:g}\ref{ass:g=lipschitz} or \assref{ass:g}\ref{ass:g=indicator}\ref{ass:g=indicator_nonempty_ri}}.
    \end{cases}
\end{equation*}
\end{corollary}

The next section explains how this smoothed information can be transferred to stronger, finite-time guarantees on the nonsmooth stationarity certificates introduced in \secref{sec:gaps} under slightly stronger assumptions.


\section{Transferring Convergence Rates from Smoothed Gaps to Nonsmooth Gaps}
\label{sec:gap_transfer}

\thmref{thm:rates} shows that the average and best smoothed gaps converge to zero at the rate $\mathcal{O}(N^{-1/4})$. However, as discussed in \secref{sec:gaps}, $\gap^{\beta_k}(x_k)$ certifies stationarity only for the smoothed problem with objective $\Phi_k=f+g^{\beta_k}\circ T$. In this section, we relate this smoothed information to stationarity certificates for the original nonsmooth problem \eqref{P}.
The analysis naturally splits into the two cases of \assref{ass:g}.

\subsection{The Indicator Case}

Throughout this subsection, we suppose that \assref{ass:g}\ref{ass:g=indicator} holds so that $g=\iota_{\D}$ with either \ref{ass:g=indicator_nonempty} or \ref{ass:g=indicator_nonempty_ri} as well. In \secref{sec:smoothed_convergence} we showed that \lemref{lem:indicator_gap_transfer} was enough to prove qualitative stationarity of smoothed gap-vanishing subsequences for \eqref{P}. To obtain explicit finite-time rates for feasibility and for the signed gap, we use the following error-bound.

\begin{assumption}[Indicator error bound]
\label{ass:indicator_error_bound}
There exists a constant $\kappa_{\mathrm{ind}}>0$ such that, for every $x\in\C$,
\begin{equation*}
    \dist_{\widetilde{\C}}(x)
    \leq
    \kappa_{\mathrm{ind}}\dist_{\D}(Tx).
\end{equation*}
\end{assumption}

\begin{proposition}\label{prop:T_surjective_error_bound}
    Suppose \assref{ass:T}, \ref{ass:C}, and \ref{ass:g}\ref{ass:g=indicator}\ref{ass:g=indicator_nonempty_ri} hold and assume $T$ is surjective. Then, \assref{ass:indicator_error_bound} holds.
\end{proposition}
\begin{proof}
    Fix $\ybar\in \ri(\D)\cap\ri(T(\C))$. Since $T(\ri(\C))=\ri(T(\C))$ under our assumptions, there exists $\xbar\in\ri(\C)$ such that $T\xbar=\ybar$. By \assref{ass:g}\ref{ass:g=indicator}\ref{ass:g=indicator_nonempty_ri}, it holds $\ri(T(\C))\cap\D\neq\varnothing$ and thus $\mathrm{range}(T)\cap\ri(\D)\neq\varnothing$ since $\ri(T(\C))\subset\mathrm{range}(T)$. From this we deduce that $T^{-1}(\ri(\D))=\ri(T^{-1}(\D))$ and therefore $\xbar\in \ri(T^{-1}(\D))$. Hence, $\ri(\C)\cap\ri(T^{-1}(\D))\neq\varnothing$. By \cite[Corollary 6]{Baus99}, $\exists \,\kappa_0>0$ such that
    \begin{equation}
      \label{e:dist1}
        \dist_{\C\cap T^{-1}(\D)}(x)\leq \kappa_0\dist_{T^{-1}(\D)}(x),\quad\forall x\in\C.
    \end{equation}
    Since $T$ is surjective, $TT^{\dagger}=\Id_{\R^m}$ with $T^\dagger$ the \emph{Moore-Penrose right inverse} $T^\dagger = T^*(TT^*)^{-1}$. Fix $x\in\C$ and let $y=\proj_{\D}(Tx)$ so $\norm{Tx-y}{}=\dist_{\D}(Tx)$. Define $\hat{x}=x+T^\dagger(y-Tx)$; then 
    \begin{equation*}
        T\hat{x}=Tx+TT^{\dagger}(y-Tx)=Tx+(y-Tx)=y\in\D
    \end{equation*}
    and so $\hat{x}\in T^{-1}(\D)$. Therefore,
    \begin{equation*}
        \dist_{T^{-1}(\D)}(x)\leq \norm{x-\hat{x}}{}=\|T^{\dagger}(y-Tx)\|\leq\|T^{\dagger}\|_\mathrm{op}\norm{y-Tx}{}=\|T^{\dagger}\|_{\mathrm{op}}\dist_{\D}(Tx).
    \end{equation*}
    Combining with the previous estimate gives
    \begin{equation*}
        \dist_{\widetilde{\C}}(x) = \dist_{\C\cap T^{-1}(\D)}(x)\leq \kappa_0\dist_{T^{-1}(\D)}(x)\leq \kappa_0\|T^\dagger\|_{\mathrm{op}}\dist_{\D}(Tx).
    \end{equation*}
\end{proof}

\begin{remark}
\assref{ass:indicator_error_bound} is a standard metric regularity condition for the system $x\in \C$, $Tx\in \D$. When $T$ is surjective, a classical sufficient condition is the relative interior qualification already present in \assref{ass:g}\ref{ass:g=indicator}\ref{ass:g=indicator_nonempty_ri}, as demonstrated in \propref{prop:T_surjective_error_bound}.
While surjectivity of $T$ together with \assref{ass:g}\ref{ass:g=indicator}\ref{ass:g=indicator_nonempty_ri} is sufficient to guarantee \ref{ass:indicator_error_bound} holds, surjectivity is not necessary. Also, note that Assumption~\ref{ass:indicator_error_bound} is not required for the qualitative subsequential stationarity result in \thmref{thm:indicator_subsequential_stationarity}.
\end{remark}

\begin{lemma}[Indicator finite-time transfer]
\label{lem:indicator_finite_time_transfer}
Suppose \assref{ass:g}\ref{ass:g=indicator}\ref{ass:g=indicator_nonempty} and \assref{ass:indicator_error_bound} hold and let $x\in\C$ and $\beta>0$.
Then
\begin{equation}
\label{eq:indicator_feasibility_bound}
    \dist_{\D}(Tx)
    \leq
    \frac{\fconstant\kappa_{\mathrm{ind}}\beta}{2}
    +
    \frac{1}{2}
    \sqrt{
        \fconstant^2\kappa_{\mathrm{ind}}^2\beta^2
        +
        4\beta\gap^\beta(x)
    }
\end{equation}
and, for the signed gap given in \defref{def:indicatorGap},
\begin{equation}
\label{eq:indicator_signed_gap_bound}
    \left|\widetilde{\gap}(x)\right|
    \leq
    \max\{
    \gap^\beta(x)
    ,
    \fconstant\kappa_{\mathrm{ind}}\dist_{\D}(Tx)\}.
\end{equation}
Moreover, 
\begin{equation*}
    \dist_{\D}(Tx)
    =
    \mathcal{O}\Big(\max\big\{\beta,\sqrt{\beta\gap^\beta(x)}\big\}\Big)
    \quad\text{and}
    \quad
     \left|\widetilde{\gap}(x)\right|
    =
    \mathcal{O}\left(\max\left\{\beta,\gap^\beta(x)\right\}\right).
\end{equation*}
\end{lemma}

\begin{proof}
For any $x\in\C$, the projection satisfies $\proj_{\widetilde{\C}}(x)\in\widetilde{\C}$ and the definition of $\widetilde{\gap}$ gives
\begin{equation}
\label{e:63i}
    \begin{aligned}
    \widetilde{\gap}(x)
    =
    \max_{s\in\widetilde{\C}}
    \ip{\nabla f(x),x-s}{}
    &\geq
    \ip{\nabla f(x),x-\proj_{\widetilde{\C}}(x)}{}\\
    &\geq
    -\norm{\nabla f(x)}{2}\norm{x-\proj_{\widetilde{\C}}(x)}{2}\\
    &\geq
    -\fconstant\dist_{\widetilde{\C}}(x)\\
    &\geq
    -\fconstant\kappa_{\mathrm{ind}}\dist_{\D}(Tx),
    \end{aligned}
\end{equation}
where the last inequality uses \assref{ass:indicator_error_bound}. Combining this lower bound with \lemref{lem:indicator_gap_transfer} yields
\begin{equation*}
    \frac{1}{\beta}\dist_{\D}^2(Tx)
    -
    \fconstant\kappa_{\mathrm{ind}}\dist_{\D}(Tx)
    \leq
    \gap^\beta(x).
\end{equation*}
Letting $\delta:=\dist_{\D}(Tx)$, this is equivalent to
$
    \delta^2-\fconstant\kappa_{\mathrm{ind}}\beta\delta-\beta\gap^\beta(x)\leq 0.
$
Since $\delta\geq0$, it is bounded above by the larger root of the quadratic, which gives \eqref{eq:indicator_feasibility_bound}.
Next, \lemref{lem:indicator_gap_transfer} gives
$
    \widetilde{\gap}(x)
    \leq
    \gap^\beta(x)$, which with \eqref{e:63i}
implies \eqref{eq:indicator_signed_gap_bound}. The asymptotic estimates follow from \eqref{eq:indicator_feasibility_bound}, \eqref{eq:indicator_signed_gap_bound}, and the inequality $\sqrt{ab}\leq\max\{a,b\}$ for $a,b\geq0$.
\end{proof}

\begin{theorem}[Indicator nonsmooth certificate rate]
\label{thm:indicator_nonsmooth_gap_rate}
Suppose \assref{ass:g}\ref{ass:g=indicator}\ref{ass:g=indicator_nonempty} and \assref{ass:indicator_error_bound} hold and let $\seq{x_k}$ be generated by \FRAMES\ with
\begin{equation*}
    \gamma_k=(k+1)^{-1/2},
    \qquad
    \beta_k=\beta_0(k+1)^{-1/4}.
\end{equation*}
For $N\geq2$, let $\kmin\in\argmin\limits_{\lfloor N/2\rfloor\leq k \leq N-1}\gap^{\beta_k}(x_k)$. Then the last-half best iterate satisfies
\begin{equation*}
    \dist_{\D}(Tx_{\kmin})
    =
    \mathcal{O}(N^{-1/4})
\end{equation*}
and, for the signed gap given in \defref{def:indicatorGap},
\begin{equation*}
    \left|\widetilde{\gap}(x_{\kmin})\right|
    =
    \mathcal{O}(N^{-1/4}).
\end{equation*}
\end{theorem}

\begin{proof}
By \corref{cor:last_half_best_iterate}, $\gap^{\beta_{\kmin}}(x_{\kmin}) = \mathcal{O}(N^{-1/4})$ and $\beta_{\kmin} = \mathcal{O}(N^{-1/4})$.
Applying \lemref{lem:indicator_finite_time_transfer} with $x=x_{\kmin}$ and $\beta=\beta_{\kmin}$ gives
\begin{equation*}
    \dist_{\D}(Tx_{\kmin})
    =
    \mathcal{O}\left(\max\left\{\beta_{\kmin},\sqrt{\beta_{\kmin}\gap^{\beta_{\kmin}}(x_{\kmin})}\right\}\right)
    =
    \mathcal{O}(N^{-1/4}).
\end{equation*}
The estimate on $\left|\widetilde{\gap}(x_{\kmin})\right|$ follows from \eqref{eq:indicator_signed_gap_bound}.
\end{proof}

\begin{remark}[Logarithmic smoothing is slower]\label{rmk:log_is_slower}
Just like how Lemma~\ref{lem:indicator_finite_time_transfer} yields the $\mathcal{O}(N^{-1/4})$ results in \thmref{thm:indicator_nonsmooth_gap_rate}, a straightforward application of Lemma~\ref{lem:indicator_finite_time_transfer} with the schedules $\beta_k=\mathcal{O}(1/\log(k+2))$ and $\gamma_k=\mathcal{O}(1/\sqrt{k})$ suggested in \cite{woodstock2025splitting,halbey2025efficient} provides a nonsmooth convergence rate of $\mathcal{O}(\log(N+2)^{-1})$ for the average absolute value of the signed gaps or the last-half best iterates $|\widetilde{\gap}(x_{k_N^*})|$, which is much slower. This is because the analysis in \cite{woodstock2025splitting,halbey2025efficient} only characterizes the effect of the smoothing schedule on the rate of convergence of the smoothed gaps to $0$ and does not take into account the difference between the smooth gap (associated to the smoothed problem) and the signed gap (associated to \eqref{P}).
\end{remark}

\subsection{The Indicator Case with Inconsistent Sets}
\label{subsec:inconsistent_indicator_case}

In the most general case of \assref{ass:g}\ref{ass:g=indicator}, $T^{-1}(\D)\cap\C$ may be empty. Then one cannot hope to show a stationarity result for \eqref{P}, since the original problem has no feasible points. Instead, the cluster points of smoothed gap-vanishing subsequences solve a best-approximate feasibility problem in the image space, namely they minimize $x\mapsto \dist_{\D}(Tx)$ over $\C$.

\begin{assumption}[Inconsistent system]
\label{ass:indicator_inconsistent}
Assumption~\ref{ass:g}\ref{ass:g=indicator} holds and
$
    \D\cap T(\C)=\varnothing.
$
\end{assumption}

Under Assumptions~\ref{ass:C}, \ref{ass:g}\ref{ass:g=indicator}, and \ref{ass:indicator_inconsistent}, compactness of $T(\C)$ and closedness of $\D$ give $\min_{x\in\C}\dist_{\D}(Tx) := \delta > 0$. Since $\C$ is bounded, \cite[Prop.~11.15, Prop.~27.1, \& Cor.~4.24]{bauschke2017convex} yield that the following \emph{closest-point set} is nonempty, compact, and convex
\begin{equation}
\label{eq:closest_point_set}
    \C^\dagger
    :=
    \argmin_{x\in\C}\tfrac{1}{2}\dist_{\D}^2(Tx)
    =
    \argmin_{x\in\C}\dist_{\D}(Tx).
\end{equation}
\begin{lemma}[Inconsistent gap-transfer]
\label{lem:inconsistent_indicator_gap_transfer}
Suppose \assref{ass:f}, \ref{ass:T}, \ref{ass:C}, \ref{ass:g}\ref{ass:g=indicator} and \ref{ass:indicator_inconsistent} hold. Let $\delta:=\min_{\xbar\in\C}\dist_{\D}(T\xbar)$, $x\in\C$, $\beta>0$, and let $C^\dagger$ be given by \eqref{eq:closest_point_set}. Then
\begin{equation}
\label{eq:inconsistent_best_approx_function_bound}
    0
    \leq
    \delta(\dist_{\D}(Tx)-\delta)
    \leq
    \dist_{\D}^2(Tx)-\delta^2
    \leq
    2\beta\gap^\beta(x)+2\beta\fconstant\dc.
\end{equation}
Moreover, for every $s\in\C^\dagger$,
\begin{equation}
\label{eq:inconsistent_secondary_one_sided_bound}
    \ip{\nabla f(x),x-s}{}\leq \gap^\beta(x).
\end{equation}
\end{lemma}

\begin{proof}
For every $z\in\C$, using $\Phi_\beta=f+\beta^{-1}(\tfrac{1}{2}\dist_{\D}^2(T\cdot))$ gives
\begin{equation*}
    \begin{aligned}
    \ip{T^*(Tx-\proj_{\D}(Tx)),x-z}{}
    &=
    \beta\ip{\nabla f(x)+\beta^{-1}T^*(Tx-\proj_{\D}(Tx)),x-z}{}
    -
    \beta\ip{\nabla f(x),x-z}{}\\
    &\leq
    \beta\gap^\beta(x)+\beta\fconstant\dc.
    \end{aligned}
\end{equation*}
Let $x^\dagger\in\C^\dagger$. By convexity of $x\mapsto \tfrac{1}{2}\dist_{\D}^2(Tx)$,
\begin{equation*}
    \dist_{\D}^2(Tx)-\delta^2
    \leq
    2\ip{T^*(Tx-\proj_{\D}(Tx)),x-x^\dagger}{}
    \leq
    2\beta\gap^\beta(x)+2\beta\fconstant\dc.
\end{equation*}
Since $\dist_{\D}(Tx)\geq\delta$, we also have
\begin{equation*}
    \delta(\dist_{\D}(Tx)-\delta)
    \leq
    (\dist_{\D}(Tx)-\delta)(\dist_{\D}(Tx)+\delta)
    =
    \dist_{\D}^2(Tx)-\delta^2,
\end{equation*}
which proves \eqref{eq:inconsistent_best_approx_function_bound}.
Now fix $s\in\C^\dagger$. By convexity of $x\mapsto \tfrac{1}{2}\dist_{\D}^2(Tx)$,
\begin{equation*}
    \tfrac{1}{2}\dist_{\D}^2(Ts)
    \geq
    \tfrac{1}{2}\dist_{\D}^2(Tx)+\ip{T^*(Tx-\proj_{\D}(Tx)),s-x}{}.
\end{equation*}
Since $\dist_{\D}(Ts)=\delta\leq \dist_{\D}(Tx)$, it follows that
$
    \ip{T^*(Tx-\proj_{\D}(Tx)),x-s}{}\geq0.
$
Therefore,
\begin{equation*}
    \begin{aligned}
        \ip{\nabla f(x),x-s}{}
        &=
        \ip{\nabla f(x)+\beta^{-1}T^*(Tx-\proj_{\D}(Tx)),x-s}{}
        -
        \beta^{-1}\ip{T^*(Tx-\proj_{\D}(Tx)),x-s}{}\\
        &\leq
        \gap^\beta(x),
    \end{aligned}
\end{equation*}
which proves \eqref{eq:inconsistent_secondary_one_sided_bound}.
\end{proof}

\begin{theorem}[Inconsistent gap-vanishing sequences]
\label{thm:inconsistent_indicator_subsequential}
Suppose \assref{ass:f}, \ref{ass:T}, \ref{ass:C}, and \ref{ass:indicator_inconsistent} hold, let $\delta:=\min_{\xbar\in\C}\dist_{\D}(T\xbar)$, $C^\dagger$ be given by \eqref{eq:closest_point_set}, and suppose $\seqj{u_j}\subset\C$ with $\seqj{\alpha_j}\subset]0,+\infty[$ define a convergent smoothed gap-vanishing sequence, i.e.,
\begin{equation*}
    \alpha_j\to0,
    \qquad u_j\to\ubar,
    \qquad\text{and}\qquad
    \gap^{\alpha_j}(u_j)\to0.
\end{equation*}
Then $\dist_{\D}(Tu_j)\to\delta$ and $\dist_{\C^\dagger}(u_j)\to0$. Moreover,
\begin{equation}
\label{eq:inconsistent_secondary_gap_convergence}
    \max_{s\in\C^\dagger}\ip{\nabla f(u_j),u_j-s}{}\to0.
\end{equation}
Consequently, every cluster point $\ubar$ of $\seqj{u_j}$ belongs to $\C^\dagger$ and satisfies
\begin{equation*}
    0\in T^*(T\ubar-\proj_{\D}(T\ubar))+N_{\C}(\ubar).
\end{equation*}
Furthermore, $\ubar$ is stationary for the secondary problem
\begin{equation*}
    \min_{x\in\C^\dagger}f(x),
\end{equation*}
i.e., $0\in\nabla f(\ubar)+N_{\C^\dagger}(\ubar)$ and, equivalently in terms of the Frank-Wolfe gap,
\begin{equation}
\label{eq:secondary_stationarity_inconsistent}
    \max_{s\in\C^\dagger}\ip{\nabla f(\ubar),\ubar-s}{}=0.
\end{equation}
In particular, if $f$ is convex on $\C^\dagger$ then $\ubar\in\argmin_{x\in\C^\dagger}f(x)$.
\end{theorem}

\begin{proof}
By \lemref{lem:inconsistent_indicator_gap_transfer},
\begin{equation*}
    0
    \leq
    \dist_{\D}^2(Tu_j)-\delta^2
    \leq
    2\alpha_j\gap^{\alpha_j}(u_j)+2\alpha_j\fconstant\dc
    \to0.
\end{equation*}
Thus $\dist_{\D}(Tu_j)\to\delta$.

We next show that $\dist_{\C^\dagger}(u_j)\to0$. Assume for the sake of contradiction that there exist $\varepsilon>0$ and a subsequence such that $\dist_{\C^\dagger}(u_j)\geq\varepsilon$. Since $u_j\to\ubar\in\C$, the subsequence also converges to $\ubar$. Continuity of $\dist_{\D}\circ T$ gives $\dist_{\D}(T\ubar)=\delta$, hence $\ubar\in\C^\dagger$, contradicting $\dist_{\C^\dagger}(u_j)\geq\varepsilon$.

For all $j\in\N$, let
\begin{equation*}
    G_j:=\max_{s\in\C^\dagger}\ip{\nabla f(u_j),u_j-s}{}.
\end{equation*}
Taking the maximum over $s\in C^\dagger$ in \eqref{eq:inconsistent_secondary_one_sided_bound} gives $G_j\leq \gap^{\alpha_j}(u_j)\to0$. For the lower bound, let $p_j:=\proj_{\C^\dagger}(u_j)$. Since $p_j\in\C^\dagger$,
\begin{equation*}
    G_j\geq \ip{\nabla f(u_j),u_j-p_j}{}\geq -\fconstant\norm{u_j-p_j}{2}=-\fconstant\dist_{\C^\dagger}(u_j)\to0.
\end{equation*}
Therefore $G_j\to0$, proving \eqref{eq:inconsistent_secondary_gap_convergence}.

Let $\ubar$ be a cluster point of $\seqj{u_j}$. Since $\dist_{\C^\dagger}(u_j)\to0$ and $\C^\dagger$ is closed, $\ubar\in\C^\dagger$. Since $\ubar$ minimizes $x\mapsto \tfrac{1}{2}\dist_{\D}^2(Tx)$ over $\C$, the first-order optimality condition gives
\begin{equation*}
    0\in T^*(T\ubar-\proj_{\D}(T\ubar))+N_{\C}(\ubar).
\end{equation*}
Passing to the limit in \eqref{eq:inconsistent_secondary_gap_convergence}, using continuity of $\nabla f$ and compactness of $\C^\dagger$, gives \eqref{eq:secondary_stationarity_inconsistent}. The normal cone inclusion $0\in\nabla f(\ubar)+N_{\C^\dagger}(\ubar)$ is the standard Frank-Wolfe gap characterization on the nonempty compact convex set $\C^\dagger$ \cite{CGsurvey}.

Finally, if $f$ is convex on $\C^\dagger$, then for every $s\in\C^\dagger$,
\begin{equation*}
    f(s)\geq f(\ubar)+\ip{\nabla f(\ubar),s-\ubar}{}\geq f(\ubar),
\end{equation*}
so $\ubar$ minimizes $f$ over $\C^\dagger$.
\end{proof}

\begin{corollary}[Last-half best iterate in the inconsistent case]
\label{cor:inconsistent_indicator_last_half}
Fix $p,q\in\left]0,1\right[$ satisfying $q<p$ and $p+q<1$ and set $\gamma_k=(k+1)^{-p}$ and $\beta_k=\beta_0(k+1)^{-q}$ with $\beta_0\in\left]0,\rho^{-1}\right[$. Suppose \assref{ass:f}, \ref{ass:T}, \ref{ass:C}, and \ref{ass:indicator_inconsistent} hold and let $\seq{x_k}$ be generated by \FRAMES. For $N\geq2$, let
\begin{equation*}
    \kmin\in\argmin_{\lfloor N/2\rfloor\leq k\leq N-1}\gap^{\beta_k}(x_k).
\end{equation*}
Then, as $N\to\infty$,
\begin{equation*}
    \dist_{\D}(Tx_{\kmin})\to\delta
    \qquad\text{and}\qquad
    \dist_{\C^\dagger}(x_{\kmin})\to0.
\end{equation*}
Moreover,
\begin{equation}
\label{eq:inconsistent_last_half_secondary_convergence}
    \max_{s\in\C^\dagger}\ip{\nabla f(x_{\kmin}),x_{\kmin}-s}{}\to0.
\end{equation}
Every cluster point of $\seq{x_{\kmin}}$ is stationary for $\min_{x\in\C^\dagger}f(x)$. If $f$ is convex on $\C^\dagger$, then every such cluster point minimizes $f$ over $\C^\dagger$.
\end{corollary}

\begin{proof}
By \corref{cor:last_half_best_iterate}, $\gap^{\beta_{\kmin}}(x_{\kmin})\to0$. Since $\kmin\geq\lfloor N/2\rfloor$ and $\beta_k\to0$, we also have $\beta_{\kmin}\to0$. The result follows from \thmref{thm:inconsistent_indicator_subsequential} applied with $u_N=x_{\kmin}$ and $\alpha_N=\beta_{\kmin}$.
\end{proof}

\begin{remark}
Under \assref{ass:indicator_inconsistent}, \FRAMES\ cannot generate a feasible limit point for the original problem \eqref{P}, because no such point exists. The conclusion above is therefore a best-approximation statement: smoothed gap-vanishing sequences approach the closest-point set $\C^\dagger$, and their cluster points are stationary for $f$ restricted to that set.
\end{remark}

\subsection{The Lipschitz Weakly Convex Case}

We now suppose that \assref{ass:g}\ref{ass:g=lipschitz} holds, making \eqref{P} the constrained analog of the setting considered in \cite{bohm2021variable} (because of $\C$). In this case, the smoothed gradient
\begin{equation*}
    \xi_\beta(x):=\nabla g^\beta(Tx)
\end{equation*}
belongs to $\partial g(\prox_{\beta g}(Tx))$, not necessarily to $\partial g(Tx)$. This prevents the smoothed gap from being directly identified with the nonsmooth gap, $\gap(x;\xi)$. The following assumption gives a finite-time transfer under an additional lifting assumption, which is similar to the surjectivity assumption made in \cite{bohm2021variable}.

\begin{assumption}[Proximal lift]
\label{ass:prox_lift}
There exists a constant $M_{\mathrm{lift}}>0$ such that, for every $x\in\C$ and every $\beta\in\left]0,\rho^{-1}\right[$, there exists $z=z(x,\beta)\in\C$ satisfying
\begin{equation*}
    Tz=\prox_{\beta g}(Tx)
\end{equation*}
and
\begin{equation*}
    \norm{z-x}{2}
    \leq
    M_{\mathrm{lift}}\norm{Tz-Tx}{2}.
\end{equation*}
\end{assumption}
Before continuing, we emphasize that \assref{ass:prox_lift} is used only for the finite-time gap-transfer bounds in \lemref{lem:lipschitz_gap_transfer} and \thmref{thm:nonsmooth_gap_rate}. The qualitative subsequential stationarity results in \thmref{thm:lipschitz_subsequential_stationarity} and \corref{cor:stationary_cluster_points} do not require it.

\begin{remark}\label{rmk:proxlift-examples}
\assref{ass:prox_lift} asks that the proximal point $\prox_{\beta g}(Tx)$ can be lifted back into $\C$ through some $z\in\C$ satisfying
\begin{equation*}
    Tz=\prox_{\beta g}(Tx)
    \qquad\text{and}\qquad
    \norm{z-x}{2}\leq M_{\mathrm{lift}}\norm{Tz-Tx}{2}.
\end{equation*}
This is an additional restriction on the triple $(g,T,\C)$ and does not hold in general. We mention two simple settings where it is satisfied.

\begin{enumerate}
    \item \emph{Identity operator with componentwise-shrinking prox.}
    Let $T=\Id$ and let $\C=\{x:\norm{x}{p}\leq\tau\}$ be an $\ell^p$-norm ball for some $p\in[1,\infty]$. If $\prox_{\beta g}$ acts componentwise and satisfies
    \begin{equation*}
        |\prox_{\beta g}(y)_i|\leq |y_i|
        \qquad
        \forall i,
    \end{equation*}
    then $z:=\prox_{\beta g}(x)$ satisfies $\norm{z}{p}\leq\norm{x}{p}\leq\tau$ so $z\in\C$. This holds for $g=\lambda\norm{\cdot}{1}$, whose proximal operator is soft-thresholding, and for the usual separable SCAD and MCP penalties, whose proximal maps shrink each coordinate toward zero. In this case one can take $M_{\mathrm{lift}}=1$.

    \item \emph{Invertible $T$ with prox-invariant image constraint.}
    Suppose that $T$ is invertible and that $\C=T^{-1}(\mathcal{K})$ for some set $\mathcal{K}\subset\R^m$ satisfying
    \begin{equation*}
        \prox_{\beta g}(\mathcal{K})\subset \mathcal{K}
        \qquad
        \forall \beta\in\left]0,\rho^{-1}\right[.
    \end{equation*}
    Then, for $x\in\C$, we have $Tx\in\mathcal{K}$ and hence $\prox_{\beta g}(Tx)\in\mathcal{K}$. Defining $z:=T^{-1}\prox_{\beta g}(Tx)$ gives $z\in\C$ and $Tz=\prox_{\beta g}(Tx)$. Moreover,
    \begin{equation*}
        \norm{z-x}{2}
        =
        \norm{T^{-1}\left(\prox_{\beta g}(Tx)-Tx\right)}{2}
        \leq
        \norm{T^{-1}}{\op}
        \norm{\prox_{\beta g}(Tx)-Tx}{2}.
    \end{equation*}
    Thus \assref{ass:prox_lift} holds with $M_{\mathrm{lift}}=\norm{T^{-1}}{\op}$.
    A simple way to guarantee prox-invariance is to take $\mathcal{K}$ to be a sublevel set of $g$, i.e., $\mathcal{K}=\{y:g(y)\leq \alpha\}$, since the definition of the proximal point gives
    \begin{equation*}
        g(\prox_{\beta g}(y))
        +
        \frac{1}{2\beta}\norm{\prox_{\beta g}(y)-y}{2}^2
        \leq
        g(y)
    \end{equation*}
    and therefore $g(\prox_{\beta g}(y))\leq g(y)$.
\end{enumerate}
\end{remark}

\begin{lemma}[Lipschitz gap-transfer]
\label{lem:lipschitz_gap_transfer}
Suppose \assref{ass:f}, \ref{ass:T}, \ref{ass:C}, \ref{ass:g}\ref{ass:g=lipschitz} and \assref{ass:prox_lift} all hold. Let $x\in\C$, $\beta\in\left]0,\rho^{-1}\right[$, and choose $z=z(x,\beta)\in\C$ as in \assref{ass:prox_lift}. Define $\xi_\beta(x):=\nabla g^\beta(Tx)$ and recall the subgradient gap given in \defref{def:nonsmoothGap}. Then $\xi_\beta(x)\in\partial g(Tz)$ and, denoting $C_{\mathrm{lift}}
    :=
    M_{\mathrm{lift}}L_g
    (
        \fconstant
        +
        L_g\norm{T^*}{\op}
        +
        \gradfconstant\dc
    )$,
\begin{equation}
\label{eq:lipschitz_gap_transfer}
    \gap(z;\xi_\beta(x))
    \leq
    \gap^\beta(x)
    +
    C_{\mathrm{lift}}\beta.
\end{equation}
\end{lemma}

\begin{proof}
Since $Tz=\prox_{\beta g}(Tx)$, the proximal optimality condition gives
\begin{equation*}
    \xi_\beta(x)
    =
    \frac{Tx-\prox_{\beta g}(Tx)}{\beta}
    \in
    \partial g(\prox_{\beta g}(Tx))
    =
    \partial g(Tz).
\end{equation*}
Therefore $\gap(z;\xi_\beta(x))$ is well-defined.

For any $s\in\C$, add and subtract $x$ and $\nabla f(x)$ to obtain
\begin{equation*}
    \begin{aligned}
    \ip{\nabla f(z)+T^*\xi_\beta(x),z-s}{}
    &=
    \ip{\nabla f(x)+T^*\xi_\beta(x),x-s}{}\\
    &\quad\quad+
    \ip{\nabla f(x)+T^*\xi_\beta(x),z-x}{}\\
    &\quad\quad+
    \ip{\nabla f(z)-\nabla f(x),z-s}{}.
    \end{aligned}
\end{equation*}
Taking the maximum over $s\in\C$ gives
\begin{equation}
\label{e:lip-gt1}
    \begin{aligned}
    \gap(z;\xi_\beta(x))
    &\leq
    \gap^\beta(x)
    +
    \norm{\nabla f(x)+T^*\xi_\beta(x)}{2}\norm{z-x}{2}
    +
    \gradfconstant\norm{z-x}{2}\dc.
    \end{aligned}
\end{equation}
By the boundedness of $\nabla f$ on $\C$ and \propref{prop:moreauProperties},
\begin{equation}
\label{e:lip-gt2}
    \norm{\nabla f(x)+T^*\xi_\beta(x)}{2}
    \leq
    \fconstant+L_g\norm{T^*}{\op}.
\end{equation}
Moreover, by \assref{ass:prox_lift} and \propref{prop:moreauDisplacementBound},
\begin{equation}
\label{e:lip-gt3}
    \norm{z-x}{2}
    \leq
    M_{\mathrm{lift}}\norm{\prox_{\beta g}(Tx)-Tx}{2}
    \leq
    M_{\mathrm{lift}}\beta L_g.
\end{equation}
Substituting \eqref{e:lip-gt2} and \eqref{e:lip-gt3} into \eqref{e:lip-gt1} yields \eqref{eq:lipschitz_gap_transfer}.
\end{proof}

\begin{theorem}[Lipschitz nonsmooth gap rate]
\label{thm:nonsmooth_gap_rate}
Suppose \assref{ass:f}, \ref{ass:T}, \ref{ass:C}, \ref{ass:g}\ref{ass:g=lipschitz}, \ref{ass:gamma}, \ref{ass:beta}, and \assref{ass:prox_lift} hold and let $\seq{x_k}$ be generated by \FRAMES\ with
\begin{equation*}
    \gamma_k=(k+1)^{-1/2},
    \qquad
    \beta_k=\beta_0(k+1)^{-1/4},
\end{equation*}
where $\beta_0\in\left]0,\rho^{-1}\right[$. For each $k\in\N$, let $z_k:=z(x_k,\beta_k)$ be chosen as in \assref{ass:prox_lift}, set $\xi_k:=\nabla g^{\beta_k}(Tx_k)$, and recall the subgradient gap given in \defref{def:nonsmoothGap}.
Then there exists a constant $C_{\mathrm{ns}}>0$ such that, for every $N\in\N^*$,
\begin{equation*}
    0
    \leq
    \frac{1}{N}\sum_{k=0}^{N-1}\gap(z_k;\xi_k)
    \leq
    C_{\mathrm{ns}}N^{-1/4}.
\end{equation*}
Consequently,
\begin{equation*}
    0
    \leq
    \min_{0\leq k\leq N-1}\gap(z_k;\xi_k)
    \leq
    C_{\mathrm{ns}}N^{-1/4}.
\end{equation*}
Moreover, for last-half best iterate with index $\kmin\in\argmin\limits_{\lfloor N/2\rfloor\leq k \leq N-1}\gap^{\beta_k}(\xk)$ as in \corref{cor:last_half_best_iterate}, it holds
\begin{equation*}
    \gap(z_{\kmin};\xi_{\kmin})
    =
    \mathcal{O}(N^{-1/4}).
\end{equation*}
\end{theorem}

\begin{proof}
By \lemref{lem:lipschitz_gap_transfer}, for every $k\in\N$,
\begin{equation*}
    \gap(z_k;\xi_k)
    \leq
    \gap^{\beta_k}(x_k)
    +
    C_{\mathrm{lift}}\beta_k.
\end{equation*}
Averaging this inequality and using \thmref{thm:rates} gives
\begin{equation*}
    \frac{1}{N}\sum_{k=0}^{N-1}\gap(z_k;\xi_k)
    \leq
    CN^{-1/4}
    +
    \frac{C_{\mathrm{lift}}}{N}
    \sum_{k=0}^{N-1}\beta_k.
\end{equation*}
Since $\beta_k=\beta_0(k+1)^{-1/4}$,
\begin{equation*}
    \sum_{k=0}^{N-1}\beta_k
    \leq
    \beta_0\left(1+\frac{4}{3}N^{3/4}\right).
\end{equation*}
Thus the average nonsmooth gap is $\mathcal{O}(N^{-1/4})$. The minimum bound follows because the gaps are nonnegative. The final statement for the last-half best iterate corresponding to $\kmin$ follows by combining \lemref{lem:lipschitz_gap_transfer} with \corref{cor:last_half_best_iterate}.
\end{proof}

\begin{remark}[Optimal balance among power schedules]
\label{rmk:optimal_power_balance_nonsmooth}
We now return to the power schedules from \thmref{thm:rates}. For $\gamma_k=(k+1)^{-p}$ and $\beta_k=\beta_0(k+1)^{-q}$ with $0<q<p$ and $p+q<1$, the smoothed gap exponent is $\min\{p-q,\ 1-p-q\}$.
The finite-time transfer estimates above show there is an additional smoothing error $\beta_k=\mathcal{O}(k^{-q})$ to account for. Thus the exponent governing the final nonsmooth stationarity certificates is
\begin{equation}\label{eq:nonsmooth_rate_exponent}
    \min\{q,\min\{p-q,\ 1-p-q\}\}
    =
    \min\{q,\;p-q,\;1-p-q\}.
\end{equation}
In the Lipschitz weakly convex case, this follows directly from \lemref{lem:lipschitz_gap_transfer}; in the indicator case, the signed gap is controlled by the same bottleneck through \lemref{lem:indicator_finite_time_transfer}. Maximizing \eqref{eq:nonsmooth_rate_exponent} over $0<q<p<1$ leads directly to the choices $p=1/2$ and $q=1/4$.
\end{remark}



\section{Applications and Numerical Experiments}\label{sec:experiments}

In this section, we apply \FRAMES\ to several problems that illustrate its flexibility across assumptions on $g$. \secref{sec:exp-splitting} considers a nonconvex splitting problem where the nonsmooth gap is computable in closed form, providing a direct verification of the convergence theory. \secref{sec:exp-nmf} and \secref{sec:exp-trend} share the same smooth objective and constraint set but differ in the nonsmooth term: the first uses a nonnegativity indicator (\assref{ass:g}\ref{ass:g=indicator}), the second uses trend filtering with weakly convex penalties (\assref{ass:g}\ref{ass:g=lipschitz}). Finally, in \secref{sec:exp-inconsistent-linf} we explore what happens for inconsistent problems where $\D\cap T(\C) = \varnothing$ as in \assref{ass:indicator_inconsistent}.
\subsection{Nonconvex Splitting over \texorpdfstring{$\ell_1$}{ℓ₁} Balls}\label{sec:exp-splitting}

We construct a problem fitting into the splitting framework of \cite{yurtsever2018conditional,silveti2020generalized,woodstock2025splitting} and for which we can compute the signed gap in closed form. We first consider
\begin{equation}\label{eq:splitting}
  \min_{x \in \C_1 \cap \C_2}
  f(x) = \tfrac{1}{2}x^\top Q x - b^\top x,
\end{equation}
where $Q \in \R^{n \times n}$ is symmetric and indefinite (so that $f$ is  nonconvex but satisfies \assref{ass:f}), $b \in \R^n$, $\C_1 = \{x : \|x - e_1\|_1 \leq 2\}$, and $\C_2 = \{x : \|x + e_1\|_1 \leq 2\}$ where $e_1$ denotes the vector with $1$ in the first entry and $0$ everywhere else. One can verify that $\C_1 \cap \C_2 = B_1 := \{x : \|x\|_1 \leq 1\}$.

To obtain an instance of \eqref{P}, we lift to the product space $\vx = (x_1, x_2) \in \C_1 \times \C_2$ and enforce consensus among components
\begin{equation*}
  \min_{(x_1, x_2) \in \C_1 \times \C_2}
  \underbrace{f(\tfrac{x_1 + x_2}{2})}_{f(\vx)}
  +
  \underbrace{\iota_{\{0\}}(x_1 - x_2)}_{g(T\vx)},
\end{equation*}
with $T\vx = x_1 - x_2$, $\D = \{0\}$, $\xbar = \tfrac{x_1+x_2}{2}$, and $\prox_{\beta g}(z) = P_{\{0\}}(z) = 0$. This satisfies \assref{ass:T}, \ref{ass:C}, and \ref{ass:g}\ref{ass:g=indicator}\ref{ass:g=indicator_nonempty_ri}. Moreover, the LMO is separable over the blocks,
\begin{equation*}
    \lmo_{\C_1\times\C_2}((G_1, G_2)) = (\lmo_{\C_1}(G_1),\ \lmo_{\C_2}(G_2)),
\end{equation*}
and easily computed using the vector sign.

\paragraph{Computable Nonsmooth Frank-Wolfe Gap}
Since $\C_1 \cap \C_2 = B_1$ is known, the feasible set of the original problem is $\tilde\C = \{(s,s) : s \in B_1\}$. The nonsmooth gap in this setting corresponds to a signed gap as given in \defref{def:indicatorGap}; at $\vx = (x_1, x_2)$ is $\widetilde{\mathrm{gap}}(\vx) = \max_{s \in \tilde\C} \ip{\nabla f(\vx), \vx - s}{}$. Since $\nabla f(\vx) = \tfrac{1}{2}(Q\xbar - b,\; Q\xbar - b)$, substituting $s = (s', s')$ with $s' \in B_1$ yields
\begin{align*}
  \widetilde{\mathrm{gap}}(\vx)
  &= \max_{s' \in B_1}
    \ip{Q\xbar - b, \xbar - s'}{}.
\end{align*}
Denoting $q:=Q\xbar-b$, this is the Frank-Wolfe gap of the original problem~\eqref{eq:splitting} evaluated at $\xbar$:
\begin{equation*}
  \widetilde{\mathrm{gap}}(\vx)
  = \langle q,\, \xbar - \lmo_{B_1}(q) \rangle.
\end{equation*}
Since $\lmo_{B_1}(q)$ returns $-\mathrm{sign}(q_{j^\star}) e_{j^\star}$ where $j^\star \in \argmax_j |q_j|$, we have $\langle q, -\lmo_{B_1}(q)\rangle = \|q\|_\infty$, giving
\begin{equation}\label{eq:ns-gap-splitting}
  \widetilde{\mathrm{gap}}(\vx)
  = \langle Q\xbar - b,\, \xbar \rangle + \|Q\xbar - b\|_\infty,
\end{equation}
which is easily computed in closed-form using the available variables while \FRAMES\ runs.

\paragraph{Data Generation}
We generate $Q = V\Lambda V^\top$ with $V$ a random orthogonal matrix and $\Lambda = \diag(\lambda_1, \ldots, \lambda_n)$ with eigenvalues drawn uniformly from $[0.5, 5]$ in magnitude, roughly $30\%$ of them negated. The vector $b$ is sampled randomly with components following the standard normal distribution. We use $n = 50$, and $N = 50,000$. The number of iterations $N$ is intentionally large to demonstrate the difference in the schedules.

\paragraph{Results}
Figure~\ref{fig:splitting} displays the convergence using \FRAMES\ with $\gamma_k=1/(k+1)^{1/2}$ and various $\beta_0 \in \{0.25, 1/L_{\nabla f}, 0.5, 1, 2, 4\}$, where $L_{\nabla f} = \|Q\|_{\mathrm{op}}/2=2.49$. We also compare the power schedule $\beta_k=\beta_0/(k+1)^{1/4}$ against the log schedule $\beta_k = \beta_0/\log(k+2)$ for each value of $\beta_0$. This is shown in Figure~\ref{fig:splitting}, where dashed lines represent the log schedule and solid lines represent the power schedule.

\begin{figure}[H]
  \centering
  \includegraphics[width=\textwidth]{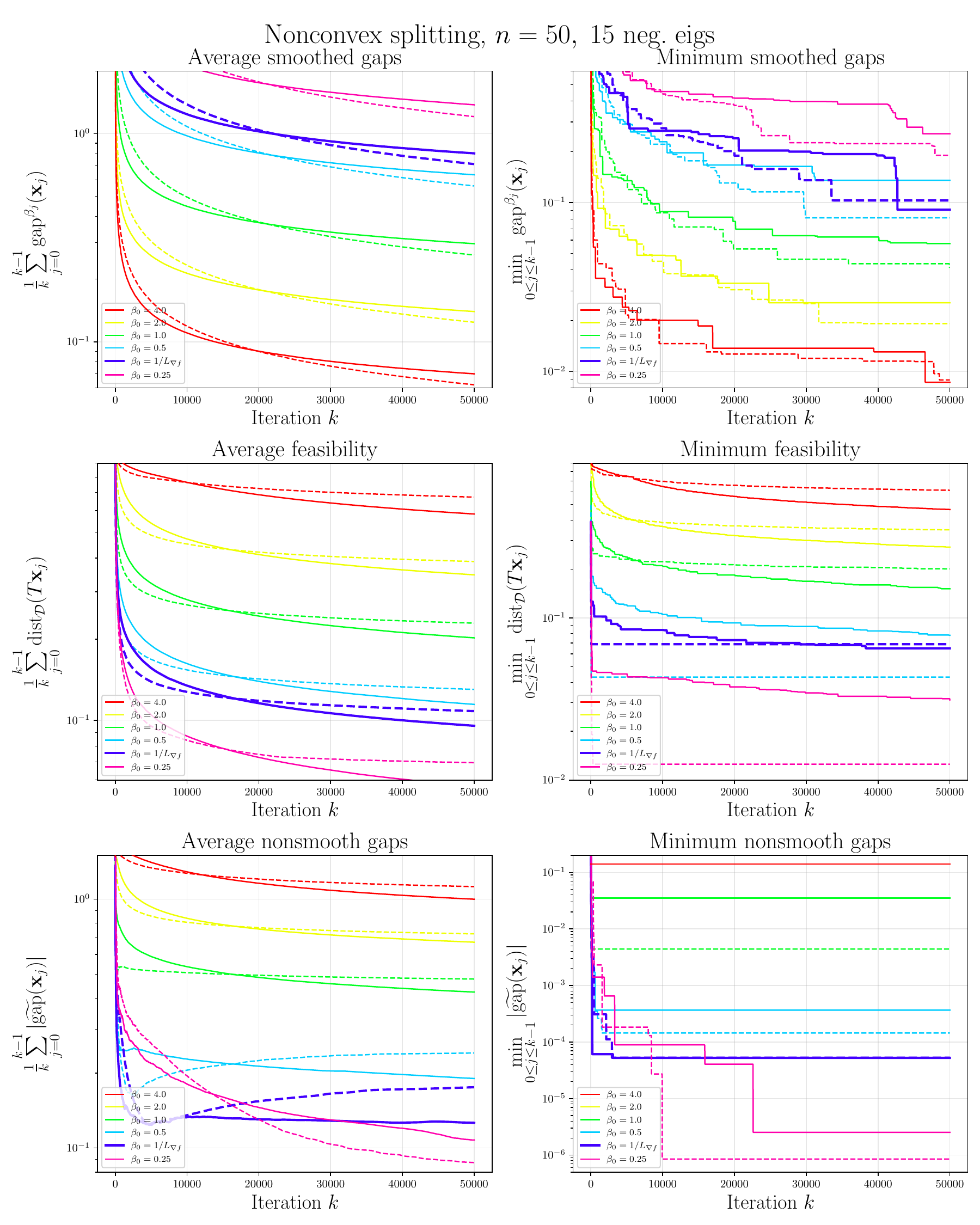}
    \caption{
        Solid curves use the power schedule $\beta_k=\beta_0/(k+1)^{1/4}$ and dashed curves use the (natural) log schedule $\beta_k=\beta_0/\log(k+2)$, as suggested by prior work \cite{woodstock2025splitting,halbey2025efficient}; both use the step size $\gamma_k = (k+1)^{-1/2}$. In the bottom right plot, the curves for $\beta_0=4$ and $\beta_0=2$ overlap.
    }
  \label{fig:splitting}
\end{figure}

\begin{itemize}
  \item \emph{Smoothed gaps.}
    The running minimum and average of the smoothed gaps decrease at the predicted $\mathcal{O}(k^{-1/4})$ rate for all values of $\beta_0$, confirming \thmref{thm:rates} in the nonconvex setting.

  \item \emph{Feasibility.}
    The consensus violation $\|x_{1,k} - x_{2,k}\|^2$ decreases for all $\beta_0$, consistent with \lemref{lem:indicator_finite_time_transfer}. Setting $\beta_0$ very small produces a trajectory which reaches the lowest infeasibility. On the other hand, the smaller $\beta_0$ is, the slower the smoothed gaps converge for both schedules.

  \item \emph{Nonsmooth gap.}
    We plot $|\widetilde{\mathrm{gap}}(\vx_k)|$ in the right panel, testing the predictions of \thmref{thm:indicator_nonsmooth_gap_rate}. When $\xbar_k \notin B_1$, the nonsmooth gap~\eqref{eq:ns-gap-splitting} can be negative. Smaller $\beta_0$ seems to yield a smaller $|\widetilde{\mathrm{gap}}|$, with $\beta_0 = 0.25$ achieving the best values for both schedules. The log schedule produces a larger average and minimal $|\widetilde{\mathrm{gap}}|$ than the power schedule for most $\beta_0$ values, confirming that the slower smoothing decay is detrimental as predicted by \thmref{thm:indicator_nonsmooth_gap_rate} and Remark~\ref{rmk:log_is_slower}.
\end{itemize}

\subsection{Low-rank Matrix Factorization}\label{sec:exp-mf}

In this section, we consider two matrix factorization problems, focusing on nonnegative matrix factorization \cite{gillis2020nonnegative} in \secref{sec:exp-nmf} and matrix factorization with trend filtering \cite{wang2016trend} in \secref{sec:exp-trend}.
Both experiments consider the smooth nonconvex objective
\begin{equation*}
  f(U,V) = \tfrac{1}{2}\|UV^\top - X^\star\|_F^2,
\end{equation*}
where $X^\star \in \R^{m \times n}$ is a given matrix and $U \in \R^{m \times r}$, $V \in \R^{n \times r}$ are matrix variables.
The constraint set is the product of two spectral-norm balls,
\begin{equation*}
  \C
  =
  \bigl\{U\in\R^{m\times r}:\norm{U}{\mathrm{op}}\leq \tau_U\bigr\}
  \times
  \bigl\{V\in\R^{n\times r}:\norm{V}{\mathrm{op}}\leq \tau_V\bigr\},
\end{equation*}
so, given a gradient $(G_U,G_V)$, the LMO decomposes across the two factor blocks, i.e.,
\begin{equation}
\label{e:lmo-lr}
    \lmo_\C((G_U,G_V)) = \left(\lmo_{\{W\colon \norm{W}{\mathrm{op}}\}\leq\tau_U}(G_U),\quad \lmo_{\{W\colon \norm{W}{\mathrm{op}}\}\leq\tau_V}(G_V)\right).
\end{equation}
Since $f$ is a quartic polynomial in the entries of $(U,V)$, it is smooth but nonconvex; restricted to the convex compact constraint set $\C$, Assumptions~\ref{ass:f}, \ref{ass:T}, and \ref{ass:C} are satisfied.

For a spectral-norm ball $\{M : \|M\|_{\mathrm{op}} \leq \tau\}$, the LMO applied to a gradient $G$ returns $-\tau \cdot \mathrm{msign}(G)$, where $\mathrm{msign}(G) = L_G R_G^\top$ is the matrix sign obtained from the reduced SVD, $G = L_G \Sigma_G R_G^\top$, by replacing all nonzero singular values with ones. The optimization variable is represented as a single vector $x = [\mathrm{vec}(U),\ \mathrm{vec}(V)] \in \R^{(m+n)r}$. The cost of computing \eqref{e:lmo-lr} using Newton-Schulz \cite{amsel2025polar} is cheap compared to the projection operator of $\C$.

\subsubsection{Nonnegativity Constraints}\label{sec:exp-nmf}

We consider the problem of factoring the matrix $X^\star$ into two rank $r=20$ factors $U$ and $V$ with nonnegative entries,
\begin{equation}\label{eq:nmf}
  \min_{\substack{\|U\|_{\mathrm{op}} \leq \tau_U \\
                   \|V\|_{\mathrm{op}} \leq \tau_V}}
  \frac{1}{2}\|UV^\top - X^\star\|_F^2
  \;+\;
  \iota_{\R_+^{m \times r} \times \R_+^{n \times r}}(U,V)
\end{equation}
which has many applications \cite{gillis2020nonnegative}. This is an instance of \eqref{P} with $T=\mathrm{Id}$ and $g=\iota_{\D}$, where
\begin{equation*}
    \D
    =
    \R_+^{m \times r}\times \R_+^{n \times r}
\end{equation*}
satisfies \assref{ass:g}\ref{ass:g=indicator}\ref{ass:g=indicator_nonempty_ri}.
Thus, for all $\beta>0$, $\prox_{\beta g}(x)=\proj_{\D}(x)=\max(x,0)$ where the maximum is applied componentwise.

\paragraph{Data Generation}
We generate nonnegative ground-truth factors $U^\star, V^\star$ with i.i.d.\ entries drawn as $|\mathcal{N}(0,1)|$, set $X^\star = U^\star (V^\star)^\top$, and fix the constraint radii as $\tau_U = 1.05\,\|U^\star\|_{\mathrm{op}}$, $\tau_V = 1.05\,\|V^\star\|_{\mathrm{op}}$, ensuring that the ground truth is strictly feasible in $\C$. Throughout, $m = n = 100$ and $N = 50,000$ iterations.

\begin{figure}[htb]
  \centering
  \includegraphics[width=\textwidth]{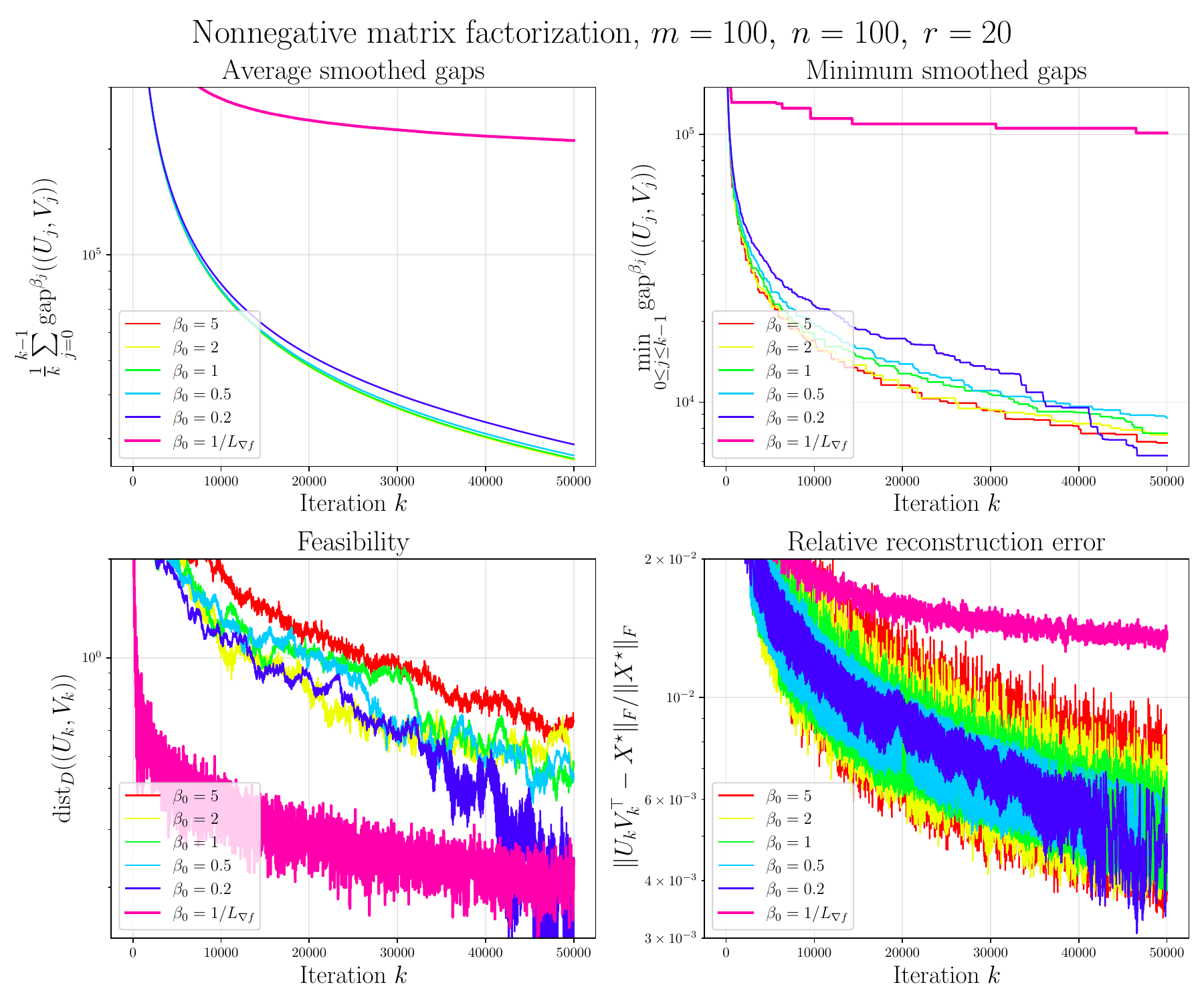}
  \caption{
    Convergence profiles for \FRAMES\ applied to nonnegative matrix factorization problem of Section~\ref{sec:exp-nmf} with $T=\mathrm{Id}$ and $g=\iota_{\D}$ for $\D=\R_+^{m\times r}\times\R_+^{n\times r}$.
  }
  \label{fig:nmf}
\end{figure}

\paragraph{Sensitivity to \texorpdfstring{$\beta_0$}{β0}}
Figure~\ref{fig:nmf} displays the behavior of \FRAMES\ across $\beta_0 \in \{1/L_{\nabla f},0.2,0.5,1,2,5\}$, where $L_{\nabla f}$ is estimated numerically to be $373.07$. We use $\gamma_k=(k+1)^{-1/2}$ and $\beta_k=\beta_0(k+1)^{-1/4}$. The value $\beta_0=1/L_{\nabla f}$ makes the initial Lipschitz constant of the Moreau-envelope term comparable to that of the smooth matrix-factorization loss.

\paragraph{Results}
The four panels in Figure~\ref{fig:nmf} show average smoothed gaps, minimum smoothed gaps, feasibility, and relative reconstruction error.
\begin{itemize}
  \item \emph{Smoothed gaps.}
    The average and running minimum of the smoothed gaps decrease as predicted in \thmref{thm:rates} for all plotted values of $\beta_0$.  Smaller values of $\beta_0$ place more weight on the Moreau envelope term and lead to larger smoothed gaps but smaller feasibility gaps, consistent with the dependence on $\beta_k^{-1}$ in \thmref{thm:rates} and \thmref{thm:indicator_nonsmooth_gap_rate}.

  \item \emph{Feasibility.}
    The feasibility panel plots $\dist_{\D}(x_k)$ for the nonnegative orthant. The smallest value, $\beta_0=1/L_{\nabla f}$, enforces nonnegativity most aggressively and gives the lowest infeasibility among the runs, while slowing down convergence of the smoothed gaps.

  \item \emph{Reconstruction error.}
    The relative error $\|U_kV_k^\top-X^\star\|_F/\|X^\star\|_F$ decreases for all values of $\beta_0$. We remark that smaller $\beta_0$ values give better feasibility and $\beta_0=0.2$ (one of the smallest $\beta_0$ values we tried) seems to perform the best. This illustrates the practical tradeoff between optimizing the smoothed objective and enforcing the nonsmooth constraint.
\end{itemize}

\subsubsection{Trend Filtering Regularization}\label{sec:exp-trend}

We now replace the nonnegativity constraint by a nonsmooth weakly convex regularizer acting on finite differences of the $U$ factor, a variant of what is commonly used in \emph{trend filtering} problems \cite{wang2016trend,kim2009ell_1,fan2022graph}. Let
$D_{\mathrm{row}}\in\R^{(m-1)\times m}$ be the first-order difference matrix,
i.e.,
\begin{equation*}
  (D_{\mathrm{row}}U)_{i,j}=U_{i+1,j}-U_{i,j}.
\end{equation*}
We consider the problem
\begin{equation}\label{eq:trend-mf}
  \min_{\substack{\|U\|_{\mathrm{op}} \leq \tau_U \\
                   \|V\|_{\mathrm{op}} \leq \tau_V}}
  \frac{1}{2}\|UV^\top-X^\star\|_F^2
  \;+\;
  \sum_{i,j}g_0\big((D_{\mathrm{row}}U)_{i,j}\big),
\end{equation}
which factors an observed matrix $X^\star$. This fits \eqref{P} with the linear map
\begin{equation*}
    T(U,V)=D_{\mathrm{row}}U,
\end{equation*}
and with $g$ being the separable sum of $g_0$ across the entries of $D_{\mathrm{row}}U$.  The effect of the regularizer is to promote piecewise-constant structure along the rows of the columns of $U$.

We compare two nonconvex penalties for $g_0$, each relying on two hyperparameters $(\lambda,a)$:
\begin{itemize}
    \item $g_0(t)=\mathrm{SCAD}_{\lambda,a}(t) = \begin{cases}\lambda|t| & \text{if }|t|\leq \lambda\\ \frac{-2a\lambda|t|+t^2+\lambda^2}{2(1-a)} & \text{if }\lambda<|t|\leq a\lambda\\ \frac{\lambda^2(a+1)}{2} & \text{if }|t|>a\lambda \end{cases}$\quad \cite{fan2001variable};
    \item $g_0(t)=\mathrm{MCP}_{\lambda,\gamma}(t)=\begin{cases}\lambda|t|-\frac{t^2}{2\gamma} & \text{if }|t|\leq \gamma\lambda\\ \frac{\gamma\lambda^2}{2} & \text{if }|t|>\gamma\lambda\end{cases}$\quad \cite{zhang2010nearly}.
\end{itemize}
Both penalties are Lipschitz and weakly convex, and both have closed-form proximity operators and formulas for $\rho$ \cite{fan2001variable,zhang2010nearly}. Hence this experiment falls under \assref{ass:g}\ref{ass:g=lipschitz}.

\begin{figure}[thb]
  \centering
  \includegraphics[width=\textwidth]{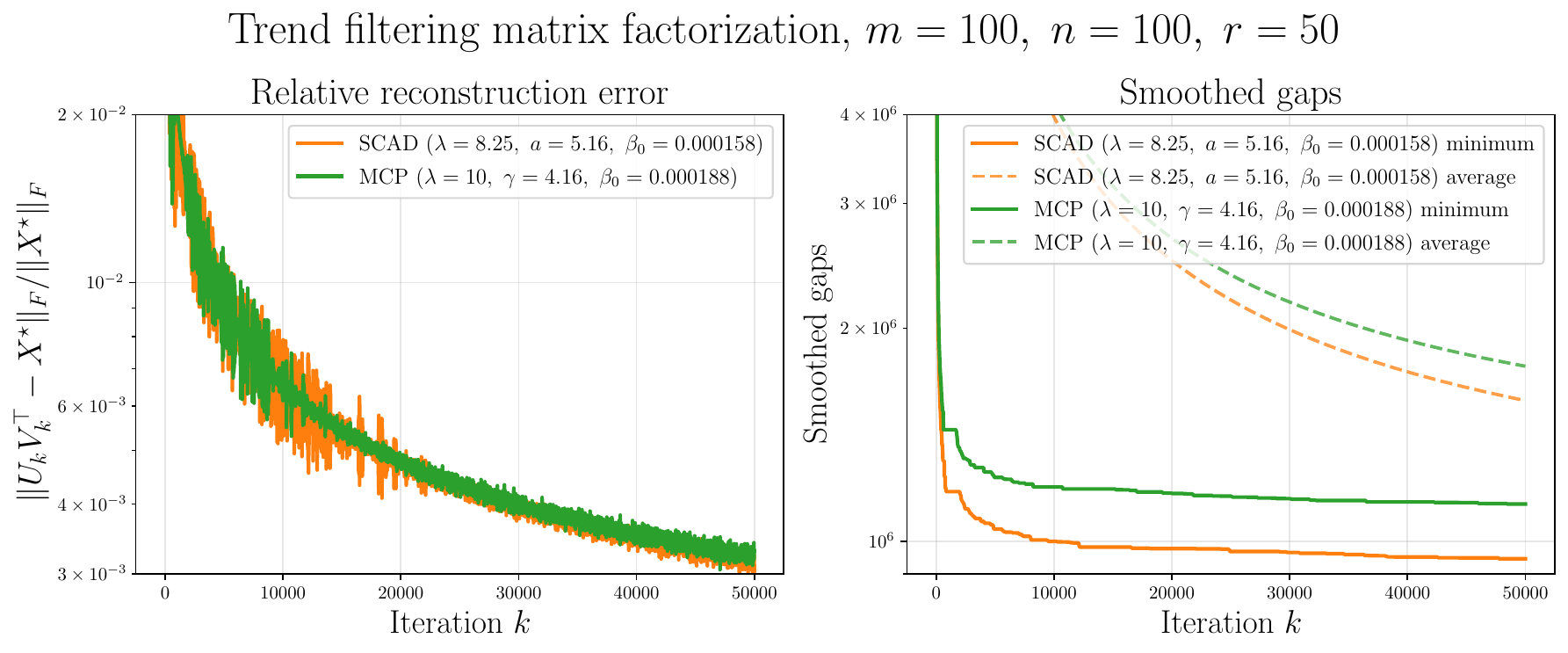}
  \caption{
    \FRAMES\ applied to trend-filtered matrix factorization problem of Section~\ref{sec:exp-trend} with tuned MCP and SCAD penalties.
  }
  \label{fig:trend_filtering}
\end{figure}

\paragraph{Data Generation}
The ground-truth factor $U^\star$ is generated with five constant blocks per column, with block heights drawn independently and block boundaries shared across columns.  Thus $D_{\mathrm{row}}U^\star$ is sparse.  The factor $V^\star$ has i.i.d.\ entries distributed as $|\mathcal{N}(0,1)|$, and $X^\star=U^\star(V^\star)^\top$.  In this experiment we use $m=n=100$, rank $r=50$, and $N=50,000$ iterations for the displayed trajectories.  The displayed SCAD run uses the hyperparameters
\begin{equation*}
    \mathrm{SCAD}:\quad
    \lambda=8.25,\quad a=5.16,\quad \beta_0=1.58\cdot 10^{-4},
\end{equation*}
and the displayed MCP run uses the hyperparameters
\begin{equation*}
    \mathrm{MCP}:\quad
    \lambda=10,\quad \gamma=4.16,\quad \beta_0=1.88\cdot 10^{-4},
\end{equation*}
which were found with a simple grid search for the plot; we have included them here for reproducibility but they are otherwise unimportant for this experiment.

\paragraph{Results}
Figure~\ref{fig:trend_filtering} compares the tuned SCAD and MCP runs.
\begin{itemize}
  \item \emph{Reconstruction error.}
    The left panel shows that both nonconvex trend-filtering penalties reduce the relative reconstruction error to the same range, with MCP slightly lower at the end of the run.

  \item \emph{Smoothed gaps.}
    The right panel plots both the running minimal smoothed gap up to iteration $k$ (solid curves) and the average of the smoothed gaps (dashed curves). The smoothed gaps decrease for both penalties at the rates predicted by \thmref{thm:rates}, with SCAD producing the smaller smoothed gaps on this instance.
\end{itemize}

\subsection{Inconsistent Problem over the
\texorpdfstring{$\ell_\infty$}{ℓ∞} Ball}
\label{sec:exp-inconsistent-linf}

\begin{wrapfigure}[19]{r}{0.5\textwidth}
    \vspace{-.875cm}
  \centering
  \includegraphics[width=0.5\textwidth]{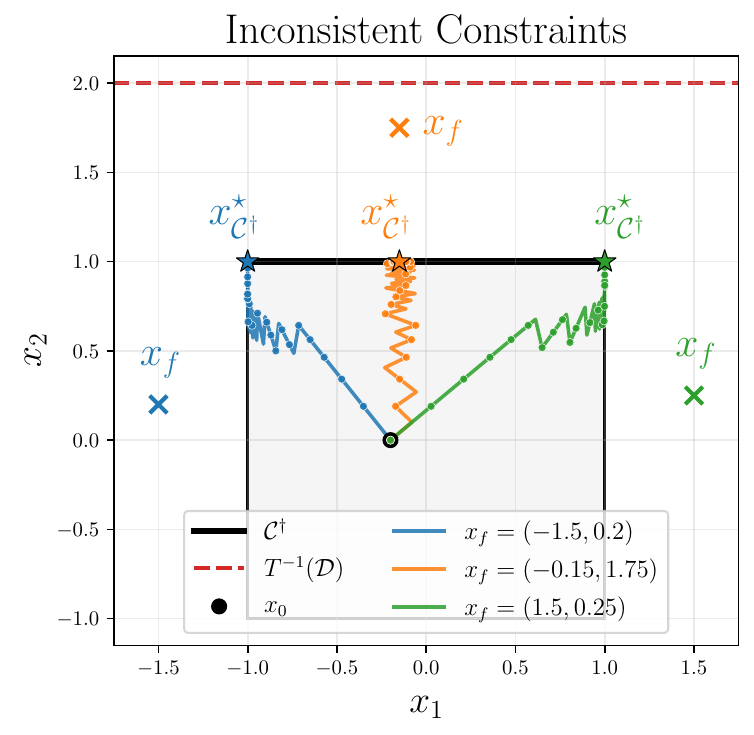}
  \caption{
    Trajectories for \FRAMES\ applied to the inconsistent problem with $\C=[-1,1]^2$, $T(x_1,x_2)=x_2$, and $\D=\{2\}$.  The dashed red line is $T^{-1}(\D)=\{(x_1,x_2):x_2=2\}$, while the black segment is the closest-point set $\C^\dagger$.  The three colored trajectories correspond to the three choices of $x_f$ in $f(x)=\norm{x-x_f}{}^2$.
  }
  \label{fig:inconsistent_linf_trajectories}
\end{wrapfigure}

We consider an example with inconsistent constraints in $\R^2$ (see \secref{subsec:inconsistent_indicator_case}). Let
\begin{equation*}
    \C=[-1,1]^2,\;
    T=\begin{bmatrix}0&1\end{bmatrix},\;
    \D=\{2\}.
\end{equation*}
Thus $T\colon\R^2\to\R$ is the coordinate projection $T(x_1,x_2)=x_2$.  Since $T(\C)=[-1,1]$, we have $T(\C)\cap\D=\varnothing$, and the indicator problem
\begin{equation*}
    \min_{x\in\C} f(x)+\iota_{\D}(Tx)
\end{equation*}
has no feasible point.  The closest-point set from
\eqref{eq:closest_point_set} is
\begin{equation}\label{eq:linf_closest_face}
    \begin{aligned}
        \C^\dagger
        &=
        \argmin_{x\in\C}\dist_{\D}(Tx)\\
        &=
        \{(x_1,1):x_1\in[-1,1]\},
    \end{aligned}
\end{equation}
and $\delta=\min_{x\in\C}\dist_{\D}(Tx)=1$.

\paragraph{Secondary Selection}
To illustrate the selection of a point in $\C^\dagger$, we use the smooth objective
\begin{equation*}
    f(x)=\norm{x-x_f}{}^2
\end{equation*}
with various choices of anchor point $x_f$: either $(-1.5,0.2)$, $(-0.15,1.75)$, or $(1.5,0.25)$. For this choice of $f$, the secondary solution is
\begin{equation*}
    x_{\C^\dagger}^\star
    \in
    \argmin_{s\in \C^\dagger}f(s)
    =
    \proj_{\C^\dagger}(x_f),
\end{equation*}
which gives $x_{\C^\dagger}^\star=(-1,1)$, $x_{\C^\dagger}^\star=(-0.15,1)$, and $x_{\C^\dagger}^\star=(1,1)$ for the three different choices of anchor point, respectively.

\paragraph{Algorithmic Details}
We run \FRAMES\ from $x_0=(-0.2,0)$ for $N=1500$ iterations with
$
    \beta_k=\beta_0(k+1)^{-1/4}$,
    $\beta_0=3$,
    and 
    $\gamma_k=(k+100)^{-1/2}$.
The shifted step size is used only to help visualize the initial part of the trajectory; the smoothing schedule is the same power schedule used throughout the experiments.  Using Theorem~\ref{thm:inconsistent_indicator_subsequential}, convergence for this schedule is also
    guaranteed using a similar argument to Corollary~\ref{cor:inconsistent_indicator_last_half}.

\paragraph{Results}
Figure~\ref{fig:inconsistent_linf_trajectories} shows that the iterates
approach $\C^\dagger$ and then select the point in $\C^\dagger$ minimizing the
smooth objective.  Figure~\ref{fig:inconsistent_linf_residuals} plots
$|\dist_{\D}(Tx_k)-\delta|$, which converges to zero because
$\dist_{\D}(Tx_k)\to\delta=1$, and
$|f(x_k)-f(x_{\C^\dagger}^\star)|$, showing convergence to the secondary
minimizer predicted by \thmref{thm:inconsistent_indicator_subsequential}.  In
this example, $f$ is convex on $\C^\dagger$, so stationarity implies
optimality.

\begin{figure}[H]
  \centering
  \includegraphics[width=\textwidth]{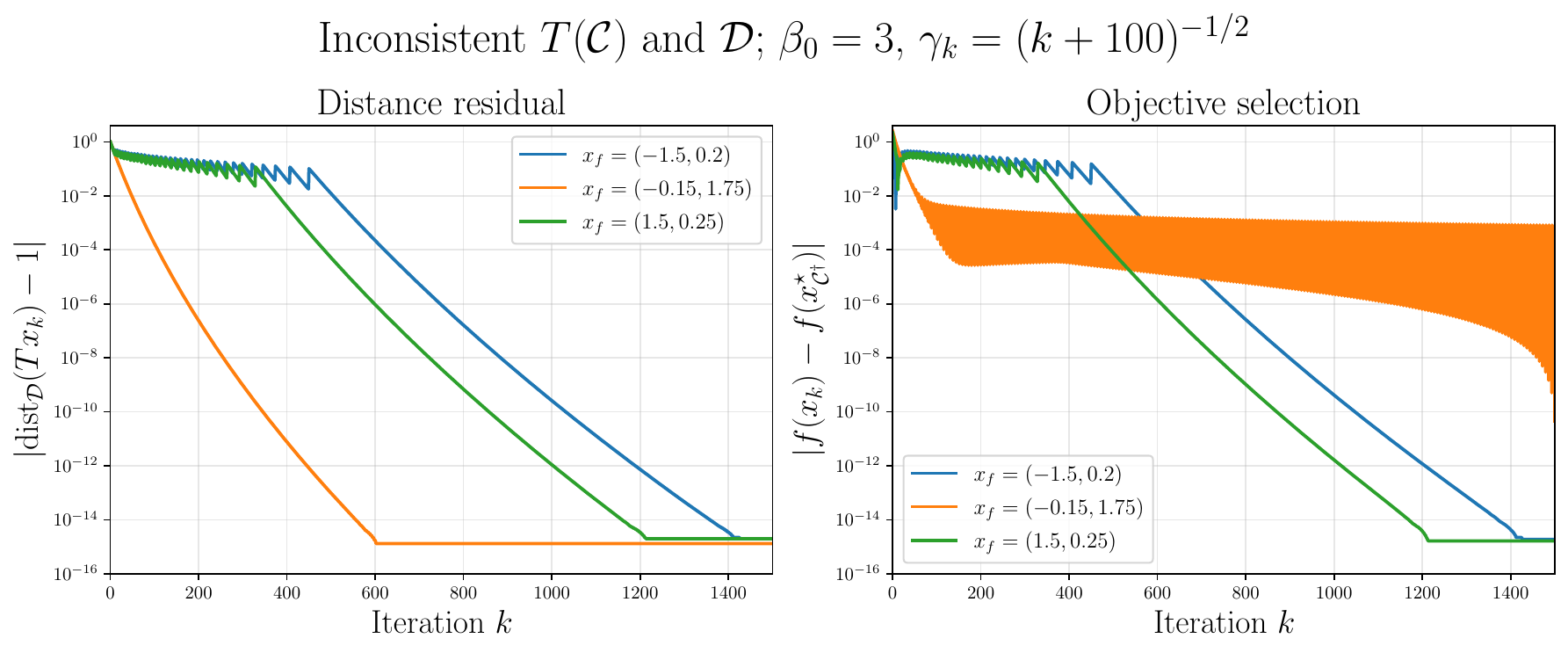}
  \caption{
    Convergence profiles for \FRAMES\ applied to the inconsistent problem where $\delta=1$.
  }
  \label{fig:inconsistent_linf_residuals}
\end{figure}

\section{Conclusion}
We have shown how to extend Frank-Wolfe to the nonsmooth nonconvex composite setting using the Moreau envelope. Our analysis reveals a delicate relationship between the step size and the smoothing parameter, with both having a strong effect on the rate of convergence for the Frank-Wolfe gap of the original problem \eqref{P}. Larger smoothing parameters make the smoothed objectives easier to optimize but give poorer approximations of the original nonsmooth problem, while smaller smoothing parameters improve the approximation at the cost of worse smoothness constants and slower convergence of the stationarity criterion for \eqref{P}. In future work, we would like to extend these ideas and their analysis to some unconstrained analogs of Frank-Wolfe, such as normalized steepest descent \cite{boyd2004convex}. We also think it would be fruitful to develop a short step analysis of nonsmooth Frank-Wolfe. A main limitation of the finite-time Lipschitz-case gap-transfer result is the proximal-lift assumption, which requires a feasible point $z\in\C$ satisfying $Tz=\prox_{\beta g}(Tx)$. Removing or weakening this assumption is an important direction for future work.


\backmatter

\bmhead{Acknowledgements}

This work was supported by the French National Research Agency (ANR) under grant ANR-25-CE23-3749 (project SIMPLES). The work of ZW was supported by the National Science Foundation under Grant DMS-2532423.

\bmhead{Competing Interests} 

Not applicable.

\bmhead{Availability of Data and Materials}

All code is available on \href{https://github.com/tonysf/frankwolfe}{GitHub}.

\bibliography{bib_mp}

@article{loca,
	Adsurl = {https://ui.adsabs.harvard.edu/abs/2019arXiv190110348L},
	Archiveprefix = {arXiv},
	Author = {{Locatello}, Francesco and {Yurtsever}, Alp and {Fercoq}, Olivier and {Cevher}, Volkan},
	Eid = {arXiv:1901.10348},
	Eprint = {1901.10348},
	Journal = {arXiv e-prints},
	Month = {Jan},
	Pages = {arXiv:1901.10348},
	Primaryclass = {math.OC},
	Title = {{Stochastic Conditional Gradient Method for Composite Convex Minimization}},
	Year = {2019}}

@article{Baus99,
  title={Strong conical hull intersection property, bounded linear regularity, {J}ameson's property ({G}), and error bounds in
convex optimization},
  author={Bauschke, Heinz H. and Borwein, Jonathan M. and Li, Wu},
  journal={Math. Program., Ser. A},
  volume={86},
  number={4},
  pages={135–160},
  year={1999},
  publisher={Springer-Verlag}
}

@article{frankwolfe,
	Author = {M. Frank and P. Wolfe},
	Journal = {Naval research logitistics quarterly},
	Number = {1-2},
	Pages = {95-110},
	Title = {An algorithm for quadratic programming},
	Volume = {3},
	Year = {1956}}

@inproceedings{lacoste2015,
	Author = {Lacoste-Julien, S. and Jaggi, M.},
	Booktitle = {NIPS},
	Pages = {496--504},
	Title = {On the global linear convergence of {F}rank-{W}olfe optimization variants},
	Year = {2015}}

@misc{prox-op.net,
	Author = {G. Chierchia and E. Chouzenoux and P. L. Combettes and J.-C. Pesquet},
	Howpublished = {https://proximity-operator.net/},
	Title = {The {P}roximity {O}perator {R}epository}}

@article{FW.jl,
  title={{FrankWolfe.jl}: A High-Performance and Flexible Toolbox for {Frank-Wolfe} Algorithms and Conditional Gradients},
  author={Besan{\c{c}}on, Mathieu and Carderera, Alejandro and Pokutta, Sebastian},
  journal={INFORMS Journal on Computing},
  year={2022},
  publisher={INFORMS}
}

@book{CGsurvey,
  title={Conditional Gradient Methods: From Core Principles to AI Applications},
  author={Braun, G{\'a}bor and Carderera, Alejandro and Combettes, Cyrille W and Hassani, Hamed and Karbasi, Amin and Mokhtari, Aryan and Pokutta, Sebastian},
  year={2025},
  publisher={SIAM},
  address = {Philadelphia}
}

@book{rockafellar1998variational,
  title={Variational Analysis},
  author={Rockafellar, R Tyrrell and Wets, Roger J-B},
  volume={317},
  year={2009},
  publisher={Springer Science \& Business Media},
  address={Berlin}
}

@article{Pier84,
  title={Decomposition through formalization in a product space},
  author={Pierra, Guy},
  journal={Math. Program.},
  volume={28},
  pages={96--115},
  year={1984},
  publisher={Springer}
}

@book{bauschke2017convex,
  title={Convex Analysis and Monotone Operator Theory in Hilbert Spaces},
  author={Bauschke, Heinz H. and Combettes, Patrick L.},
  edition={2},
  year={2017},
  publisher={Springer},
  address={Cham},
  series={CMS Books in Mathematics},
  isbn={978-3-319-48311-5},
  doi={10.1007/978-3-319-48311-5}
}

@article{woodstock2025splitting,
  title={Splitting the conditional gradient algorithm},
  author={Woodstock, Zev and Pokutta, Sebastian},
  journal={SIAM Journal on Optimization},
  volume={35},
  number={1},
  pages={347--368},
  year={2025},
  publisher={SIAM}
}

@article{argyriou2014hybrid,
  title={Hybrid conditional gradient-smoothing algorithms with applications to sparse and low rank regularization},
  author={Argyriou, Andreas and Signoretto, Marco and Suykens, Johan},
  journal={Regularization, Optimization, Kernels, and Support Vector Machines},
  pages={53--82},
  year={2014},
  publisher={CRC Press Boca Raton, FL}
}

@inproceedings{yurtsever2019conditional,
  title={A conditional-gradient-based augmented {L}agrangian framework},
  author={Yurtsever, Alp and Fercoq, Olivier and Cevher, Volkan},
  booktitle={International Conference on Machine Learning},
  pages={7272--7281},
  year={2019},
  organization={PMLR}
}

@inproceedings{yurtsever2018conditional,
  title={A conditional gradient framework for composite convex minimization with applications to semidefinite programming},
  author={Yurtsever, Alp and Fercoq, Olivier and Locatello, Francesco and Cevher, Volkan},
  booktitle={International conference on machine learning},
  pages={5727--5736},
  year={2018},
  organization={PMLR}
}

@article{silveti2020generalized,
  title={Generalized conditional gradient with augmented {L}agrangian for composite minimization},
  author={Silveti-Falls, Antonio and Molinari, Cesare and Fadili, Jalal},
  journal={SIAM Journal on Optimization},
  volume={30},
  number={4},
  pages={2687--2725},
  year={2020},
  publisher={SIAM}
}

@article{bohm2021variable,
  title={Variable smoothing for weakly convex composite functions},
  author={B{\"o}hm, Axel and Wright, Stephen J},
  journal={Journal of optimization theory and applications},
  volume={188},
  number={3},
  pages={628--649},
  year={2021},
  publisher={Springer}
}

@inproceedings{halbey2025efficient,
  year = {2025},
  booktitle = {Proceedings of the Conference on Neural Information Processing Systems},
  month = sep,
  volume = {38},
  archiveprefix = {arXiv},
  eprint = {2506.02635},
  arxiv = {arXiv:2506.02635},
  primaryclass = {math.OC},
  author = {Halbey, Jannis and Rakotomandimby, Seta and Besançon, Mathieu and Designolle, Sébastien and Pokutta, Sebastian},
  title = {Efficient Quadratic Corrections for Frank-Wolfe Algorithms},
  date = {2025-06-03}
}

@article{mazanti2025nonsmooth,
  title={A nonsmooth {F}rank--{W}olfe algorithm through a dual cutting-plane approach},
  author={Mazanti, Guilherme and Moquet, Thibault and Pfeiffer, Laurent},
  journal={Journal of Optimization Theory and Applications},
  volume={207},
  number={2},
  pages={29},
  year={2025},
  publisher={Springer}
}

@article{fan2001variable,
  title={Variable selection via nonconcave penalized likelihood and its oracle properties},
  author={Fan, Jianqing and Li, Runze},
  journal={Journal of the American statistical Association},
  volume={96},
  number={456},
  pages={1348--1360},
  year={2001},
  publisher={Taylor \& Francis}
}

@book{boyd2004convex,
  title={Convex optimization},
  author={Boyd, Stephen and Vandenberghe, Lieven},
  year={2004},
  publisher={Cambridge university press},
  address = {Cambridge}
}

@book{clarke1990optimization,
  title={Optimization and nonsmooth analysis},
  author={Clarke, Frank H},
  year={1990},
  publisher={SIAM},
  address={Philadelphia}
}

@article{hoheisel2010proximal,
  title={On proximal point-type algorithms for weakly convex functions and their connection to the backward euler method},
  author={Hoheisel, Tim and Laborde, Maxime and Oberman, Adam},
  journal={Optimization Online},
  year={2010}
}

@article{guelat1986some,
  title={Some comments on Wolfe's ‘away step’},
  author={Gu{\'e}lat, Jacques and Marcotte, Patrice},
  journal={Mathematical Programming},
  volume={35},
  number={1},
  pages={110--119},
  year={1986},
  publisher={Springer}
}

@inproceedings{combettes2020boosting,
  title={Boosting {F}rank--{W}olfe by chasing gradients},
  author={Combettes, Cyrille and Pokutta, Sebastian},
  booktitle={International conference on machine learning},
  pages={2111--2121},
  year={2020},
  organization={PMLR}
}

@article{combettes2021complexity,
  title={Complexity of linear minimization and projection on some sets},
  author={Combettes, Cyrille W and Pokutta, Sebastian},
  journal={Operations Research Letters},
  volume={49},
  number={4},
  pages={565--571},
  year={2021},
  publisher={Elsevier}
}

@article{woodstock2026high,
  title={High-precision linear minimization is no slower than projection},
  author={Woodstock, Zev},
  journal={Optimization Letters},
  pages={1--7},
  year={2026},
  publisher={Springer}
}

@article{bolte2024iterates,
  title={The iterates of the {F}rank--{W}olfe algorithm may not converge},
  author={Bolte, J{\'e}r{\^o}me and Combettes, Cyrille W and Pauwels, Edouard},
  journal={Mathematics of Operations Research},
  volume={49},
  number={4},
  pages={2565--2578},
  year={2024},
  publisher={INFORMS}
}

@article{silveti2021inexact,
  title={Inexact and stochastic generalized conditional gradient with augmented {L}agrangian and proximal step},
  author={Silveti-Falls, Antonio and Molinari, Cesare and Fadili, Jalal},
  journal={Journal of Nonsmooth Analysis and Optimization},
  volume={2},
  number={Original research articles},
  year={2021},
  publisher={Episciences. org}
}

@inproceedings{thekumparampil2020optimal,
  title={Optimal nonsmooth {F}rank--{W}olfe method for stochastic regret minimization},
  author={Thekumparampil, Kiran Koshy and Jain, Prateek and Netrapalli, Praneeth and Oh, Sewoong},
  booktitle={12th Annual Workshop on Optimization for Machine Learning},
  year={2020}
}

@article{thekumparampil2020projection,
  title={Projection efficient subgradient method and optimal nonsmooth {F}rank--{W}olfe method},
  author={Thekumparampil, Kiran K and Jain, Prateek and Netrapalli, Praneeth and Oh, Sewoong},
  journal={Advances in neural information processing systems},
  volume={33},
  pages={12211--12224},
  year={2020}
}

@article{amsel2025polar,
  title={The polar express: Optimal matrix sign methods and their application to the muon algorithm},
  author={Amsel, Noah and Persson, David and Musco, Christopher and Gower, Robert M},
  journal={arXiv preprint arXiv:2505.16932},
  year={2025}
}

@book{gillis2020nonnegative,
  title={Nonnegative matrix factorization},
  author={Gillis, Nicolas},
  year={2020},
  publisher={SIAM},
  address = {Philadelphia}
}

@article{wang2016trend,
  title={Trend filtering on graphs},
  author={Wang, Yu-Xiang and Sharpnack, James and Smola, Alexander J and Tibshirani, Ryan J},
  journal={Journal of Machine Learning Research},
  volume={17},
  number={105},
  pages={1--41},
  year={2016}
}

@inproceedings{fan2022graph,
  title={Graph trend filtering networks for recommendation},
  author={Fan, Wenqi and Liu, Xiaorui and Jin, Wei and Zhao, Xiangyu and Tang, Jiliang and Li, Qing},
  booktitle={Proceedings of the 45th international ACM SIGIR conference on research and development in information retrieval},
  pages={112--121},
  year={2022}
}

@article{kim2009ell_1,
  title={$\ell_1$ trend filtering},
  author={Kim, Seung-Jean and Koh, Kwangmoo and Boyd, Stephen and Gorinevsky, Dimitry},
  journal={SIAM review},
  volume={51},
  number={2},
  pages={339--360},
  year={2009},
  publisher={SIAM}
}

@article{ravi2019deterministic,
  title={A deterministic nonsmooth {F}rank--{W}olfe algorithm with coreset guarantees},
  author={Ravi, Sathya N and Collins, Maxwell D and Singh, Vikas},
  journal={Informs Journal on Optimization},
  volume={1},
  number={2},
  pages={120--142},
  year={2019},
  publisher={INFORMS}
}

@inproceedings{cheung2018solving,
  title={Solving Separable Nonsmooth Problems Using {F}rank--{W}olfe with Uniform Affine Approximations.},
  author={Cheung, Edward and Li, Yuying},
  booktitle={IJCAI},
  pages={2035--2041},
  year={2018}
}

@article{asgari2024nonsmooth,
  title={Nonsmooth projection-free optimization with functional constraints},
  author={Asgari, Kamiar and Neely, Michael J},
  journal={Computational Optimization and Applications},
  volume={89},
  number={3},
  pages={927--975},
  year={2024},
  publisher={Springer}
}

@article{kreimeier2024frank,
  title={On a {F}rank--{W}olfe approach for abs-smooth functions},
  author={Kreimeier, Timo and Pokutta, Sebastian and Walther, Andrea and Woodstock, Zev},
  journal={Optimization Methods and Software},
  pages={1--27},
  year={2024},
  publisher={Taylor \& Francis}
}

@article{de2023short,
  title={Short Paper-A note on the {F}rank--{W}olfe algorithm for a class of nonconvex and nonsmooth optimization problems},
  author={De Oliveira, Welington},
  journal={Open Journal of Mathematical Optimization},
  volume={4},
  pages={1--10},
  year={2023}
}

@article{lan2016conditional,
  title={Conditional gradient sliding for convex optimization},
  author={Lan, Guanghui and Zhou, Yi},
  journal={SIAM Journal on Optimization},
  volume={26},
  number={2},
  pages={1379--1409},
  year={2016},
  publisher={SIAM}
}

@article{ouyang2023universal,
  title={Universal conditional gradient sliding for convex optimization},
  author={Ouyang, Yuyuan and Squires, Trevor},
  journal={SIAM Journal on Optimization},
  volume={33},
  number={4},
  pages={2962--2987},
  year={2023},
  publisher={SIAM}
}

@article{ito2023parameter,
  title={A parameter-free conditional gradient method for composite minimization under H{\"o}lder condition},
  author={Ito, Masaru and Lu, Zhaosong and He, Chuan},
  journal={Journal of Machine Learning Research},
  volume={24},
  number={166},
  pages={1--34},
  year={2023}
}

@article{zhang2010nearly,
  title={Nearly unbiased variable selection under minimax concave penalty},
  author={Zhang, Cun-Hui},
  year={2010}
}

\end{document}